\def\mystyle{}

\if\mystyle l
\documentclass[landscape]{amsart}
\else
\documentclass{amsart}
\fi

\usepackage{begnac}

\DeclareMathOperator{\dist}{dist}

\begin{document}

\title{Almost indiscernible sequences and convergence of canonical bases}

\author{Itaï Ben Yaacov}

\address{Itaï \textsc{Ben Yaacov} \\
  Université Claude Bernard -- Lyon 1 \\
  Institut Camille Jordan, CNRS UMR 5208 \\
  43 boulevard du 11 novembre 1918 \\
  69622 Villeurbanne Cedex \\
  France}
\urladdr{\url{http://math.univ-lyon1.fr/~begnac/}}

\author{Alexander Berenstein}
\address{Alexander Berenstein, Universidad de los Andes \\
  Cra 1 No 18A-10 \\
  Bogotá, Colombia}
\urladdr{\url{http://matematicas.uniandes.edu.co/~aberenst}}

\author{C. Ward Henson}
\address{C. Ward Henson \\
  University of Illinois at Urbana-Champaign \\
  Urbana, Illinois 61801 \\
  USA}
\urladdr{\url{http://www.math.uiuc.edu/~henson}}

\thanks{Research supported by CNRS-UIUC exchange programme, ANR chaire d'excellence junior THEMODMET (ANR-06-CEXC-007), and NSF grants DMS-0100979, DMS-0140677 and DMS-0555904.}

\svnInfo $Id: SFB.tex 1604 2013-08-06 15:58:43Z begnac $
\thanks{\textit{Revision} {\svnInfoRevision} \textit{of} \today}


\begin{abstract}
  We give a model-theoretic account for several results regarding sequences of random variables appearing in Berkes \& Rosenthal \cite{Berkes-Rosenthal:AlmostExchangeableSequences}.
  In order to do this,
  \begin{itemize}
  \item We study and compare three notions of convergence of types in a stable theory: logic convergence, i.e., formula by formula, metric convergence (both already well studied) and convergence of canonical bases.
    In particular, we characterise $\aleph_0$-categorical stable theories in which the last two agree.
  \item We characterise sequences which admit almost indiscernible sub-sequences.
  \item We apply these tools to $ARV$, the theory (atomless) random variable spaces.
    We characterise types and notions of convergence of types as conditional distributions and weak/strong convergence thereof, and obtain, among other things, the Main Theorem of Berkes \& Rosenthal.
  \end{itemize}
\end{abstract}

\maketitle

\tableofcontents

\section*{Introduction}

The main motivation for the present paper is to give a formal model-theoretic account for several probability theory results of Berkes \& Rosenthal
\cite{Berkes-Rosenthal:AlmostExchangeableSequences}.
These results have a strong model theoretic flavour to them: for example, the use of limit tail algebras (canonical bases of limit types), reference of exchangeable sequences (indiscernible sequences), distribution realisation (type realisation), compactness of the distribution space (type space compactness), and so on.

The appropriate model theoretic setting for this analysis is the continuous logic theory $ARV$ of random variables over atomless probability spaces, which is exposed in some detail in \cite{BenYaacov:RandomVariables}.
In \autoref{cor:MomentQE} we show that, modulo $ARV$, every formula $\varphi(\bar x)$ can be expressed as a continuous combination of expectations of moments $\bE[\bar x^\alpha]$, so a reader not totally at ease with continuous logic may simply take this to be the definition of a formula.
Similarly, types in this theory correspond to conditional distributions, and each of the notions of convergence of conditional distributions considered by Berkes \& Rosenthal has a corresponding notion of convergence of types.
It is easy to check that \emph{weak convergence} of distributions corresponds to convergence in the logic topology (which is indeed the weakest natural topology on a type space).
We also show that \emph{strong convergence} of distributions corresponds to metric convergence of types, as well as to canonical base convergence which we define below.
Modulo these translations, the main theorem of \cite{Berkes-Rosenthal:AlmostExchangeableSequences} has a clear model theoretic counterpart, regarding existence of almost indiscernible sequences, which we prove (in a general model-theoretic setting) in \autoref{sec:AlmostIndiscernible}.

\medskip

In \autoref{sec:Convergence} we consider three topologies on the space of types of a stable theory:
\begin{enumerate}
\item The logic topology is the weakest topology we consider (since it is compact, it is minimal among Hausdorff topologies).
\item The canonical base topology is defined in terms of convergence of the canonical bases of the types.
  It is stronger than the logic topology, and over a model it is strictly stronger.
\item The metric topology is defined in terms of convergence of realisations of types.
  It is the strongest of the three.
\end{enumerate}
In \autoref{sec:SFB} we introduce SFB (strongly finitely based) theories, namely, theories for which the two last topologies agree.
In particular, we prove a useful criterion for SFB under the assumption of $\aleph_0$-categoricity.
\begin{thm*}
  A stable theory $T$ is $\aleph_0$-categorical and SFB if and only if the theory $T_P$ of lovely pairs of models of $T$ (as per Poizat \cite{Poizat:Paires}) is $\aleph_0$-categorical.
\end{thm*}
It follows easily that several familiar continuous theories, such as those of Hilbert spaces, probability algebras and random variable spaces, are SFB.

\autoref{sec:AlmostIndiscernible} is fairly independent from the preceding sections, building up to \autoref{thm:GenBR}, which is the general model-theoretic counterpart of Berkes \& Rosenthal \cite[Theorem~2.4]{Berkes-Rosenthal:AlmostExchangeableSequences}.

The theory $ARV$, of random variable spaces over atomless probability spaces, is discussed in two steps.
First, in \autoref{sec:RandomVariables} we discuss some general properties.
We characterise types as conditional distributions, and show that the logic topology agrees with weak convergence.
We also start proving that strong convergence of conditional distributions agrees with distance and canonical base convergence of types: we show that the former lies between the two latter ones; once we show, in \autoref{sec:SFB}, that $ARV$ is SFB, it follows that all three agree.
Second, in \autoref{sec:BR} we put everything together, showing that several of Berkes \& Rosenthal's results, including their main theorem, are special cases of model theoretic ones.

\medskip

Throughout this paper we assume that $T$ is a stable continuous theory.
We assume that the reader is familiar with basic facts regarding stability and continuous logic, as presented in \cite{BenYaacov-Usvyatsov:CFO}.
We diverge slightly from the conventions of these references, in that we do not distinguish between formulae and definable predicates, and refer to all as just ``formulae'' (one may consider that by ``formula'' here we mean a ``limit formula'' in the sense of \cite{BenYaacov-Usvyatsov:CFO}).

For material regarding the theory $ARV$ we refer the reader to \cite{BenYaacov:RandomVariables}.
Other background material includes Poizat \cite{Poizat:Paires} for beautiful pairs and Pillay \cite{Pillay:GeometricStability} (Chapter 2, Sections 4 and 5, specifically Theorem 5.12) for Zilber's Theorem and its consequences for $\aleph_0$-categorical strongly minimal and $\aleph_0$-stable theories.

\section{Convergence of types and canonical bases}
\label{sec:Convergence}

As said earlier, we work throughout in the context of a fixed theory $T$, which, when necessary, is assumed to be stable.
Since we shall be manipulating types throughout the paper, let us say a few words about them.
Let $X$ be an arbitrary set, let $\ell^\infty(X)$ denote the Banach space of bounded complex functions on $X$, and let $\cF \subseteq \ell^\infty(X)$.
Then we have a natural evaluation map $e\colon X \rightarrow \bC^\cF$, and $\overline{e(X)}$ is a compact Hausdorff space which can be naturally identified with the maximal ideal space of the sub-$C^*$-algebra generated by $\cF$ in $\ell^\infty(X)$.
Let us denote this space by $\beta_\cF(X)$.
In out setting, every $m$-ary formula $\varphi(\bar x)$ (i.e., formula with $m$ free variables $\bar x = (x_0,\ldots,x_{m-1})$) defines a bounded function on $M^m$ for each model $M \vDash T$, and we construct the space $m$-types in $T$ as
\begin{gather*}
  \tS_m(T) = \beta_\cF(M^m), \qquad \cF = \text{all $m$-ary formulae}.
\end{gather*}
This does not depend on the choice of $M$.
(We cheat a little -- this holds when $T$ is complete, otherwise we need to replace $M^m$ with a disjoint union of $m$-fold powers of models of all completions of $T$, and again, the choice of models is not important.)
The map $e$ will then be denoted $\tp$: $\bar a \in M^n$, its type is $\tp(\bar a) = e(\bar a) \in \tS_m(T)$.
When $p = \tp(\bar a)$ we also write $\bar a \vDash p$ and use the notation $\varphi(\bar x)^{p(\bar x)} = \varphi^p = \varphi(\bar x)$ for the evaluation map.

We shall also (or mostly) consider types over a parameter set, namely a subset $A \subseteq M$ in some model $M \vDash T$.
We then construct the space of $m$-types over $A$ as
\begin{gather*}
  \tS_m(A) = \beta_\cF(M^m), \qquad \cF = \text{all $m$-ary formulae with parameters in $A$}.
\end{gather*}
Again, this does not change if we replace $M$ with an elementary extension, and we write $\tp(\bar a/A) = e(\bar a)$.

When $T$ eliminates quantifiers we may replace ``formulae'' with ``quantifier-free formulae'' or even ``atomic formulae''.
Thus, for example, when $T = ARV$ (which eliminates quantifiers), elements of a model are $[0,1]$-valued random variables, and $\cF$ can be equivalently taken to be the family of $\bE[t(\bar x,\bar c)]$ where $t$ is a continuous function and $\bar c \in A^k$ for some $k$.
It is then not difficult to check that $\tS_m(A)$ can be identified with the space of $m$-dimensional joint conditional distributions with respect to $\sigma(A)$ equipped with the topology of weak convergence (we shall discuss all this in detail in \autoref{sec:RandomVariables}).

Also, when $T$ is a classical theory, i.e., when all atomic formulae are $\{0,1\}$-valued, we may restrict $\cF$ to classical, i.e., $\{0,1\}$-valued, formulae, without changing the end result, and we get the classical totally disconnected type spaces.

Given the parameters $A \subseteq M$, we can always replace $M$ with an elementary extension $N \succeq M$ such that $\tp(\cdot/A)\colon N^{2m} \rightarrow \tS_{2m}(A)$ is onto (all types over $A$ are realised in $N$).
We then define a distance on $\tS_m(A)$ by
\begin{gather*}
  d(p,q) = \min \, \bigl\{ d(\bar a,\bar b) \colon \bar a,\bar b \in N^m, \, \bar a \vDash p \text{ and } \bar b \vDash q \bigr\}.
\end{gather*}
The distance between two finite tuples is defined as the maximum of the distances between coordinates.
Since all formulae are uniformly continuous, this metric on $\tS_m(A)$ is stronger than the topology defined above, often called the \emph{logic topology}.

Going back to our two examples, in $ARV$ metric convergence agrees with strong convergence of joint conditional distributions, while in classical logic, the metric is discrete, and a convergent sequence must be eventually constant.

Our assumption that the theory $T$ is stable gives rise to yet another notion of convergence of types (to be more precise, this is
a notion of convergence of parallelism classes).
Recall from \cite{BenYaacov-Usvyatsov:CFO} that for every formula $\varphi(\bar x,\bar y)$ there exists a formula $d_{\bar x} \varphi(\bar y,Z)$, where $Z = (\bar z_n)_{n\in \bN}$ consists of countably many copies of $\bar x$, such that for every type $p(\bar x)$ over a model $M$ admits a $\varphi$-definition which is an instance $d_{\bar x} \varphi(\bar y,C)$:
\begin{gather*}
  \varphi(\bar x,\bar b)^p = d_{\bar x} \varphi(\bar b,C), \qquad \forall \bar b \in M^{|\bar y|}.
\end{gather*}
Moreover, if $N \succeq M$ is any elementary extension, $p$ admits a unique extension to a type over $N$ with the same definitions.

With a slight abuse of terminology, say that $A \subseteq M$ is \emph{algebraically closed} if $\acl^{eq}(A) = \dcl^{eq}(A)$.
This is an unavoidable technical condition which, in the cases of interest to us, will turn out to be quite benign: every set of random variables (in a model of $ARV$) is algebraically closed, and similarly every subset of a Hilbert space is algebraically closed.
When $A \subseteq M$ is algebraically closed, every $p(\bar x) \in \tS_m(A)$ admits a unique extension to a type over $M$ whose definitions are over $A$ (i.e., are equivalent to some formula with parameters in $A$, which need not be of the form $d_{\bar x} \varphi(\bar y,C)$), and we refer to these as being the definitions of $p$ (this canonical extension is called the \emph{non forking extension} of $p$ to $M$).
Even more generally, a type over an arbitrary set is \emph{stationary} if it has a unique extension to $\acl^{eq}(A)$ (and $A$ is algebraically closed if and only if all $m$-types over $A$, for all $m$, are stationary).

Let $S_{\Cb_\varphi}$ be the sort of canonical parameters of instances $d_{\bar x} \varphi(\bar y,Z)$.
The key property of this sort is that it is equipped with a natural metric: if $c$ and $c'$ are the canonical parameters of two instances $d_{\bar x}\varphi(\bar y,C)$ and $d_{\bar x}\varphi(\bar y,C')$, respectively, then
\begin{gather*}
  d(c,c') = \sup_{\bar y}\, \bigl| d_{\bar x}\varphi(\bar y,C)-d_{\bar x}\varphi(\bar y,C') \bigr|.
\end{gather*}
Now, for a type $p$ over an algebraically closed $A$ we define its \emph{$\varphi$-canonical base}, denoted $\Cb_\varphi(p)$, as the canonical parameter of the definition $d_{\bar x} \varphi(\bar y,C)$.
Thus, if $M \vDash T$ and $p(\bar x),q(\bar x) \in \tS_m(M)$, then:
\begin{align*}
  d\bigl( \Cb_\varphi(p),\Cb_\varphi(q) \bigr) & = \sup_{\bar b \in M}\, \bigl| \varphi(\bar x,\bar b)^p - \varphi(\bar x,\bar b)^q \bigr|.
\end{align*}

\begin{ntn}
  \label{ntn:Phi}
  For each $m$ we let $\Phi_m$ denote a set of formulae $\varphi(\bar x,\bar y)$, where $|\bar x| = m$, which is generating in the sense that every formula $\psi(\bar x,\bar y)$ (with $|\bar x| = m$) is a continuous combination of formulae in $\Phi_m$.
  Since we assume that the language is countable, we may take $\Phi_m$ to be countable.
\end{ntn}

The \emph{canonical base} of $p(\bar x) \in \tS_m(A)$ is defined as:
\begin{gather*}
  \Cb(p) = \bigl( \Cb_\varphi(p) \bigr)_{\varphi \in \Phi_m}.
\end{gather*}
The choice of $\Phi_m$ is of no importance, so long as it is generating as required in \autoref{ntn:Phi}: if $\psi$ is a continuous combination of $(\varphi_n) \subseteq \Phi_m$, then the $\psi$-definition of any $p$ can be recovered uniformly from the family of its $\varphi_n$-definitions.
We may therefore make the following convenient assumption:

\begin{conv}
  From now on we shall consider that $d_{\bar x} \varphi$ takes the canonical base as parameter: $\varphi(\bar x,\bar b)^p = d_{\bar x} \varphi(\bar b,C)$ where $C = \Cb(p)$ as above.
\end{conv}

The canonical base is usually viewed as a mere set (i.e., the minimal set to which $p$ has a non forking stationary restriction), but we will rather view it as an infinite tuple indexed by $\Phi_m$, living in the infinite sort $S_{\Cb_m} = \prod_{\varphi\in \Phi_m} S_{\Cb_\varphi}$ which only depends on $m$ (compare with \cite{BenYaacov:UniformCanonicalBases}).
Since we took $\Phi_m$ to be countable, the sort $S_{\Cb_m}$ consists of countable tuples.
As such, it is naturally equipped with a metric by enumerating $\Phi_m = \{\varphi_n\}_{n \in \bN}$ and letting
\begin{gather}
  \label{eq:CbMet}
  d\bigl( \Cb(p),\Cb(q) \bigr) = \bigvee_n 2^{-n} \wedge d\bigl( \Cb_{\varphi_n}(p),\Cb_{\varphi_n}(q) \bigr).
\end{gather}
Up to uniform equivalence, this does not depend on the chosen enumeration.
Now, convergence of canonical bases is pointwise convergence:
\begin{gather*}
  \Cb\bigl( p_n(\bar x) \bigr) \to \Cb\bigl( p(\bar x) \bigr)
  \quad \Longleftrightarrow \quad
  \Cb_\varphi(p_n) \to \Cb_\varphi(p) \text{ for all } \varphi \in \Phi_m.
\end{gather*}
We shall call this topology on $\tS_m(A)$ (where $A$ is algebraically closed) the \emph{canonical base topology}.

Types and type spaces of infinite tuples can be constructed in much the same manner.
Let $I$ be some index set, $\bar x = (x_i)_{i \in I}$.
Of course, only finitely many variables can actually appear in a formula, but we shall still call an $I$-ary formula one all of whose free variables appear in $\bar x$, and write it as $\varphi(\bar x)$ (the other variables are ``dummy''), and similarly for formulae with parameters in a set $A$.
This already gives us the logic topology on $\tS_I(A)$, and when $A$ is algebraically closed, the canonical base topology as well.

The metric topology on $\tS_I(A)$ when $I$ is infinite is a little trickier.
We observe that as a set, $\tS_I(A)$ can be naturally presented as the projective limit of $\bigl\{ \tS_{I_0}(A) : I_0 \subseteq I \text{ finite} \bigr\}$, and that for each of the logic or canonical base topologies, this is a topological inverse limit.
We therefore also define the distance topology on $\tS_I(A)$ as the inverse limit of the distance topologies on $\bigl\{ \tS_{I_0}(A) : I_0 \subseteq I \text{ finite} \bigr\}$.

\begin{rmk}
  \label{rmk:CountableMetricTopology}
  When $I$ is countable we can define a metric on $I$-tuples by identifying $I$ with $\bN$ and letting
  \begin{gather*}
    d(\bar a,\bar b) = \bigvee_{n \in \bN} 2^{-n} \wedge d(a_n,b_n).
  \end{gather*}
  This is a definable metric, and up to uniform equivalence does not depend on the enumeration of $I$, so the induced (product) uniform structure is canonical.
  Moreover, it induces the metric topology on $\tS_I(A)$ defined above.

  In addition, when $I$ is countable (or finite), we have $p_n \rightarrow p$ in the metric topology if and only if there are realisations $\bar a_n \vDash p_n$ and $\bar a \vDash p$ in an elementary extension of $M$ (the model containing $A$) such that $\bar a_n \rightarrow \bar a$ (coordinate-wise, or equivalently, in the metric on $N^I$).
\end{rmk}

\begin{ntn}
  The three topologies defined on $\tS_I(A)$ will be denoted $\sT_\cL$ (logic) $\sT_\Cb$ (canonical base) and $\sT_d$ (metric).
  For convergence of nets (or sequences) of types in these topologies we shall use the notation $p_j \rightarrow^\Box p$ or $p = \lim^\Box p_j$ where $\Box \in \{\cL, \Cb, d\}$.
  We allow ourselves to omit $\cL$ (the logic topology being ``the'' topology).
\end{ntn}

\begin{lem}
  \label{lem:TypeTopologies}
  For arbitrary theory $T$ and set $A$ we have $\sT_d \supseteq \sT_\cL$.
  When $T$ is stable and $A$ is algebraically closed we have $\sT_d \supseteq \sT_\Cb \supseteq \sT_\cL$.
\end{lem}
\begin{proof}
  It is enough to prove this when $I$ is finite, say $I = m$.
  Then the first assertion holds since all formulae are uniformly continuous.

  For the second assertion, since $I$ is finite, $\sT_d$ is metric, for $\sT_d \supseteq \sT_\Cb$ it is enough to show that for sequences, if $p_n \rightarrow^d p$ then $p_n \rightarrow^\Cb p$.
  Fix a formula $\varphi(\bar x,\bar y)$ with $|\bar x| = m$.
  For each $\varepsilon > 0$ there exists $\delta > 0$ such that for all $\bar a,\bar a',\bar b$, if $d(\bar a,\bar a') < \delta$ then $\bigl| \varphi(\bar a,\bar b) - \varphi(\bar a',\bar b) \bigr| < \varepsilon$, and for all $n$ big enough we have $d(p_n,p) < \delta$.
  For such $n$ we can choose realisations $\bar a \vDash p$, $\bar a' \vDash p_n$ such that $d(\bar a,\bar a') < \delta$, and moreover, we may choose them in such a manner that $\bar a\bar a' \ind_A M$.
  This just means that $\tp(\bar a/M)$ and $\tp(a'/M)$ are the non forking extensions of $p$ and $p_n$, respectively.
  It now follows that $d\bigl( \Cb_\varphi(p_n), \Cb_\varphi(p) \bigr) \leq \varepsilon$.
  Thus $\Cb_\varphi(p_n) \rightarrow \Cb_\varphi(p)$ for all such formulae $\varphi$, so indeed $p_n \rightarrow^\Cb p$.

  The canonical base topology is also metrisable (only the logic topology need not be, if $A$ is uncountable), so for $\sT_\Cb \supseteq \sT_\cL$ we may assume we have a sequence $p_n \rightarrow^\Cb p$.
  Then a formula over $A$ can be written as $\varphi(\bar x,\bar b)$, and we have
  \begin{gather*}
    \varphi(\bar x,\bar b)^{p_n}
    =
    d_{\bar x} \varphi(\bar b,\Cb_\varphi(p_n))
    \rightarrow
    d_{\bar x} \varphi(\bar b,\Cb_\varphi(p))
    =
    \varphi(\bar x,\bar b)^p.
  \end{gather*}
  Therefore $p_n \rightarrow p$ as desired.
\end{proof}

\begin{rmk}
  \label{rmk:CbConvergenceDistanceToModel}
  Let $M \vDash T$ be a model, and let $p_n \rightarrow^\Cb p$ in $\tS_m(M)$.
  Let $\varphi(\bar x,\bar y) = \bigvee_{i<m} d(x_i,y_i)$ be the distance formula, let $\psi = d_{\bar x} \varphi$, and let $c_n = \Cb_\varphi(p_n)$, $c = \Cb_\varphi(p)$.
  Let also $d(p,M)$ denote the distance from some (any) realisation of $p$ to $M^m$, and similarly for $p_n$.
  Then $d(p_n,M) = \inf_{\bar y} \, \psi(\bar y,c_n) \rightarrow \inf_{\bar y} \, \psi(\bar y,c) = d(p,M)$.
  In particular, a sequence of realised types can never converge in canonical base to a non realised type.

  On the other hand, the realised types over $M$ are dense in $\tS_m(M)$ in the logic topology.
  Therefore, if $M$ is non compact, so non realised types exist, we have a proper inclusion $\sT_\cL \subsetneq \sT_\Cb$.
\end{rmk}

\begin{exm}
  As per the previous Remark, examples of sequences which converge logically but not in $\Cb$ are plenty.
  Consider, for example, $\bN$ as a model of the theory of the infinite set (without extra structure).
  Let $p_n = \tp(n/\bN)$ and let $q \in \tS_1(\bN)$ be the unique non algebraic type.
  Then $p_n \rightarrow q$.
  Let $\varphi(x,y)$ be the formula $x = y$ and let $c_n$ be the canonical parameter for the $\varphi$-definition of $p_n$.
  Then the $c_n$ are all distinct (have distance one), so the sequence $(c_n)$ does not converge.
\end{exm}

A classical example where $\sT_\Cb$ differs from $\sT_d$ cannot be both $\aleph_0$-stable and $\aleph_0$-categorical (see \autoref{prp:ZilberEquivs}).
Since there is no know natural example of an $\aleph_0$-categorical, strictly stable classical theory (one can be produced using a Hrushovski construction), we shall give a non $\aleph_0$-categorical one, and a continuous one.

\begin{exm}
  Let $T = ACF_0$ be the (complete, $\aleph_0$-stable) theory of algebraically closed fields of characteristic zero.
  Let $K \vDash T$ be any model.
  For $n \in \bN$, the polynomial $X^n+Y$ is irreducible in $K[X,Y]$, and therefore gives rise to a complete type $p_n \in \tS_2(K)$.
  Similarly, let $p \in \tS_2(K)$ correspond to the trivial ideal.
  Then $p_n \rightarrow^\Cb p$.
  On the other hand, the distance on $\tS_2(K)$ is discrete, so $p_n \not\rightarrow^d p$.
\end{exm}

\begin{exm}
  \label{exm:LpNotSFB}
  Let $ALpL$ be the theory of atomless $L^p$ Banach lattices for $p \in [1,\infty)$ (see \cite{BenYaacov-Berenstein-Henson:LpBanachLattices}).
  Let $X = Y = Z = [0,1]$ with the Lebesgue measure, let $M = L^p(X) \subseteq L^p(X \times Y) \subseteq N = L^p(X \times Y \cup Z)$, where the first inclusion is induced by the projection $X \times Y \rightarrow X$, and the second by extension by zeroes.
  For each $n$ let $f_n = n^{1/p} \cdot \bone_{X \times [1-1/n,1]} \in N$, and let $p_n = \tp(f_n/M)$.
  Similarly, let $f = \bone_Z$, $p = \tp(f/M)$.
  First of all, it is clear that $p_n \not\rightarrow^d p$, and we claim that $p_n \rightarrow^\Cb p$.
  For this we shall use the characterisation of uniform canonical bases for $1$-types in $ALpL$ given in \cite[Section~3]{BenYaacov:UniformCanonicalBases}.

  For each $n$ and $t \in [0,1]$, let
  \begin{gather*}
    f_{n,t} =
    \begin{cases}
      0 & 0 \leq t \leq (n-1)/n \\
      n^{1/p} \cdot \bone_X & (n-1)/n < t \leq 1.
    \end{cases}
  \end{gather*}
  Then $f_{n,t} \in M$ increases with $t$, and $f_n$ is just $(x,y) \mapsto f_{n,y}(x)$ extended by zeroes to $X \times Y \cup Z$.
  In the notation of \cite{BenYaacov:UniformCanonicalBases} we have $\bE_{[t,s]}[f_n|M] = \int_t^s f_{n,r} \, dr$ for $0<t<s<1$.
  Similarly, $\bE_{[t,s]}[f|M] = 0$.
  In particular, $\bE_{[t,s]}[f_n|M]$ is zero for $n$ big enough, so $\bE_{[t,s]}[f_n|M] \rightarrow \bE_{[t,s]}[f|M]$ for all $0 < t < s <1$.
  In addition, $\|f^+\| = \|f_n^+\| = 1$, $\|f^-\| = \|f_n^-\| = 0$.
  By \cite[Theorem~3.16]{BenYaacov:UniformCanonicalBases}, $\Cb(f_n/M) \rightarrow \Cb(f/M)$, i.e., $p_n \rightarrow^\Cb p$.
\end{exm}

\section{The theory of $[0,1]$-valued random variables}
\label{sec:RandomVariables}

The main aim of this paper is to place results of Berkes \& Rosenthal \cite{Berkes-Rosenthal:AlmostExchangeableSequences} in a model-theoretic context.
One convenient way to code probability spaces as model-theoretic objects is via the corresponding spaces of $[0,1]$-valued random variables.
Let us recall a few facts from \cite[Section~2]{BenYaacov:RandomVariables} regarding such spaces.
Let $\Omega$ be a probability space, and $M = L^1(\Omega,[0,1])$ the space of all $[0,1]$-valued random variables, equipped with the $L^1$ distance.
Formally, we view $M$ as a metric structure $(M,0,\neg,\half,\dotminus)$ where the function symbols $\neg$, $\half$ and $\dotminus$ are interpreted naturally by composition.
We shall also use $E(X)$ as an abbreviation for $d(X,0)$, namely the expectation of $X$.
The class of all such structures is elementary, axiomatised by a universal theory $RV$.
The restriction to the operations $\neg$, $\half$ and $\dotminus$ is purely technical and may be ignored: by the lattice version of the Stone-Weierstraß Theorem, if $\theta\colon [0,1]^\alpha \rightarrow [0,1]$ is any continuous function then the map $\bar X \mapsto \theta(\bar X)$ is uniformly approximated by expressions in these symbols, and is therefore uniformly definable in all models of $RV$.

The probability algebra associated with $\Omega$ can be identified with the set of all characteristic functions in $L^1(\Omega,[0,1])$, and this set is uniformly quantifier-free definable in models of $RV$, and will be denoted by $\sF$.
For $A \subseteq M$, let $\sigma(A) \subseteq \sF^M$ denote the minimal complete sub-algebra with respect to which every $X \in A$ is measurable (so in particular $\sigma(M) = \sF^M$).

The theory $RV$ admits a model companion $ARV$, whose models are the spaces of the form $L^1(\Omega,[0,1])$ where $\Omega$ is atomless.
The theory $ARV$ is $\aleph_0$-categorical (whereby complete), $\aleph_0$-stable and it eliminates quantifiers.
Furthermore, non forking in models of $ARV$ coincides with probabilistic independence.
In other words, $A \ind_B C$ if and only if $\bP[ X | \sigma(BC) ] = \bP[ X | \sigma(B) ]$ for every $X \in \sigma(A)$ (or, equivalently, for every $X \in \sigma(AB)$).
In terms of definability of types: $B$ and $BC$ are always algebraically closed, and $\tp(A/BC)$ is definable with parameters in $B$ (equivalently, its definitions agree with those of $\tp(A/B)$) if and only if $\bP[ X | \sigma(BC) ] = \bP[ X | \sigma(B) ]$ for every $X \in \sigma(A)$.

The theories $RV$ and $Pr$ (the theory of probability algebras) are biïnterpretable.
Indeed we have already mentioned that the probability algebra is definable in the corresponding random variable space.
Conversely, using a somewhat more involved argument, one can interpret, in a probability algebra $\sF$, the space of random variables $L^1(\sF,[0,1])$, such that for $M \vDash Pr$ and $N \vDash RV$:
\begin{gather*}
  M = \sF^{L^1(M,[0,1])},
  \qquad
  N = L^1\bigl( \sF^N  ,[0,1] \bigr).
\end{gather*}

\begin{dfn}
  Let $\sA$ be a probability algebra.
  An \emph{$n$-dimensional distribution} over $\sA$ is an $L^1(\sA,[0,1])$-valued Borel probability measure $\vec \mu$ on $\bR^n$ ($\sigma$-additive in the $L^1$ topology, and $\vec \mu(\bR^n)$ is the constant function $1 \in L^1(\sA,[0,1])$).
  The space of all $n$-dimensional distributions over $\sA$ will be denoted $\fD_{\bR^n}(\sA)$.
  For a Borel set $B \subseteq \bR^n$, we denote by $\fD_B(\sA)$ the space of $n$-dimensional conditional distributions which, as measures, are supported by $B$ (we shall only use this notation for $B = [0,1]^n$).

  Let $\bar X$ be an $n$-tuple of real-valued random variables.
  The joint conditional distribution of $\bar X$ over $\sA$ denoted here by $\vec \mu = \dist(\bar X|\sA)$ (and by $c\cdot(\sA)\dist(\bar X)$ in \cite{Berkes-Rosenthal:AlmostExchangeableSequences}) is the $n$-dimensional distribution over $\sA$ given by
  \begin{gather*}
    \vec \mu(B) = \bP[\bar X \in B| \sA],
    \qquad
    B \subseteq \bR^n \text{ Borel}.
  \end{gather*}
\end{dfn}

Recall that a net $(X_i)_{i \in I} \subseteq L^1(\sA,[0,1])$ converges in the \emph{weak topology} to $X$ if for every $Y \in L^1(\sA,[0,1])$, $E[X_iY] \to E[XY]$.
The net $(X_i)$ converges to $X$ in the \emph{strong topology} if it converges in $L^1$.

\begin{dfn}
  Following \cite[Proposition~1.8]{Berkes-Rosenthal:AlmostExchangeableSequences}, say that a net $(\vec \mu_i)_{i \in I}$ of $n$-dimensional distributions over $\sA$ \emph{converges weakly (strongly)} to $\vec \mu$ if for every continuous function $\theta\colon \bR^n \to [0,1]$ we have $\int \theta(\bar x)\,d \vec \mu_i(\bar x)
  \to \int \theta(\bar x)\,d \vec \mu(\bar x)$ weakly (strongly).
\end{dfn}

Let us make two remarks regarding this last condition.
As we said earlier, if $\theta\colon [0,1]^m \to [0,1]$ is continuous then the map $\bar X \mapsto \theta(\bar X)$ is uniformly definable in models of $RV$.
Second, by the Stone-Weierstraß Theorem, every continuous $\theta$ can be arbitrarily well approximated by polynomials.
It follows that it is enough to consider only monomial test functions $\bar x^\alpha = \prod x_i^{\alpha_i}$, where $\alpha \in \bN^m$.

\begin{thm}
  \label{thm:TypesAreCondDist}
  Let $\bar X$ be an $m$-tuple in a model of $ARV$, $A$ a set, $\sA = \sigma(A)$.
  Then the joint conditional distribution $\dist(\bar X|\sA)$ depends only on $\tp(\bar X/A)$.
  Moreover, the map
  \begin{gather*}
    \zeta\colon \tp(\bar X/A) \mapsto \dist(\bar X|\sA)
  \end{gather*}
  is a homeomorphism between $\tS_m(A)$ (equipped with the logic topology) and $\fD_{[0,1]^m}(\sA)$ equipped with the topology of weak convergence.
\end{thm}
\begin{proof}
  The first assertion, as well as the injectivity of $\zeta$, are shown in \cite{BenYaacov:RandomVariables}.
  Let $\Omega$ be the Stone space of the underlying Boolean algebra of $\sA$.
  This is a compact, totally disconnected space, and $\sA$ is canonically identified with the algebra of clopen sets there.
  Let $\fD_0 \subseteq \fD_{[0,1]^m}(\sA)$ consist of all those $\vec \mu$ such that, for some finite partition $\{B_i\}_{i<k}$ of $\Omega$, and for all Borel $C \subseteq [0,1]^m$, the function $\vec \mu(C)$ is constant on each $B_i$.
  In other words, $\vec \mu \in \fD_0$ can be written as $\sum \bone_{B_i} \mu_i$ where each $\mu_i$ is an ordinary Borel probability measure on $[0,1]^m$.

  First, we claim that $\fD_0$ is dense in $\fD_{[0,1]^m}(\sA)$.
  Indeed, for $\vec \mu \in \fD_{[0,1]^m}(\sA)$, $Y \in L^1(\sA,[0,1])$, $\alpha \in \bN^m$ and $\varepsilon > 0$ let
  \begin{gather*}
    U_{\vec \mu, Y, \alpha, \varepsilon} = \left\{ \vec \nu : \left| \bE\left[ Y \int \bar x^\alpha\, d\vec \mu(\bar x) \right] - \bE\left[ Y \int \bar x^\alpha\, d\vec \nu(\bar x) \right] \right| < \varepsilon \right\}.
  \end{gather*}
  A weak neighbourhood $U$ of $\vec \mu$ always contains a finite intersection $\bigcap_{i<k} U_{\vec \mu, Y_i, \alpha_i, 3\varepsilon}$.
  For each $Y_i$ find a stair function $Z_i$ such that $|Z_i-Y_i| < \varepsilon$, so $U$ contains $\bigcap_{i<k} U_{\vec \mu, Z_i, \alpha_i, \varepsilon}$.
  Let $\{B_j\}_{j < \ell}$ be a finite partition of $\Omega$ on which each $Z_i$ is constant, and let $\mu_j$ be the average of $\vec \mu$ on $B_j$: $\mu_j(C) = \frac{\bE[ \bone_{B_j} \vec \mu(C) ]}{\bP[B_j]}$.
  Then $\vec \nu = \sum \bone_{B_j} \mu_j \in U \cap \fD_0$.

  Second, we claim that every $\vec \mu = \sum_{i<k} \bone_{B_i} \mu_i \in \fD_0$ (where $\{B_i\}$ is a partition of $\Omega$) lies in the image of $\zeta$.
  Indeed, let $\Omega' = \Omega \times [0,1]^m$, and define a probability Borel measure $\Omega'$ by
  \begin{gather*}
    \nu(C) = \sum_{i<k} (\bP \times \mu_i)\bigl( C \cap (B_i \times [0,1]^m) \bigr).
  \end{gather*}
  Clearly, the projection on the first component $\Omega' \rightarrow \Omega$ is measure-preserving, so $M = L^1\bigl( (\Omega',\nu), [0,1] \bigr)$ is a model of $RV$ which contains (a copy of) $A$, and we may embed $M \subseteq N \vDash ARV$.
  Let $\bar X \colon \Omega' \rightarrow [0,1]^m$ be the projection on the second component.
  Then $\bar X \in M^m \subseteq N^m$ and $\vec \mu = \dist(\bar X|\sA) = \zeta \tp(\bar X/A)$.

  Third, we claim that $\zeta$ is continuous.
  Indeed, let $Y \in L^1(\sA,[0,1]) = \dcl(A)$ and $\alpha \in \bN^m$.
  Since the map $\bar X \mapsto \bar X^\alpha$ is uniformly definable, the map $\bar X \mapsto E[Y\bar X^\alpha]$ is an definable by a formula over $A$, which will be denoted $E[ Y \bar x^\alpha ]$.
  If $p = \tp(\bar X/A)$ and $\vec \mu = \dist(\bar X|\sA) = \zeta p$ then
  \begin{gather*}
    E[ Y\bar x^\alpha]^{p(\bar x)} = E[ Y\bar X^\alpha] = \bE\left[ Y \int \bar x^\alpha\, d\vec \mu(\bar x) \right].
  \end{gather*}
  Thus the map $p \mapsto \bE\left[ Y \int \bar x^\alpha\, d(\zeta p)(\bar x) \right]$ is continuous in $p$, and by definition of weak convergence, $\zeta$ is continuous.

  Since $\tS_m(A)$ is compact and $\fD_{[0,1]^m}(\sA)$ Hausdorff, and given our three claims, $\zeta$ is a homeomorphism.
\end{proof}

From this point onwards we identify $m$-types over $A$ with $m$-dimensional conditional distributions over $\sigma(A)$.
In particular, from now on we shall omit $\zeta$ from the notation, writing $\int dp(\bar x)$ where before we wrote $\int d(\zeta p)(\bar x)$.
Strong convergence of conditional distributions also has a model theoretic counterpart.

\begin{cor}[Quantifier Elimination to Moments]
  \label{cor:MomentQE}
  Modulo the theory $ARV$, the $m$-ary formulae are exactly the (possibly infinite) continuous combinations of the $E[\bar x^\alpha]$.
  In particular, every formula can be approximated arbitrarily well by finite continuous combinations of these.
\end{cor}
\begin{proof}
  By the theorem, the map $p \mapsto \bigl( E[\bar x^\alpha]^p \bigr)_{\alpha \in \bN^m}$ is a topological embedding $\iota\colon \tS_m(\emptyset) \hookrightarrow [0,1]^{\bN^m}$.
  If $\varphi(\bar x)$ is any formula, then it can be identified with a continuous function $\varphi\colon \tS_m(\emptyset) \to [0,1]$, which, by Tietze's Extension Theorem, can be written as $\hat \varphi \circ \iota$ for some continuous $\hat \varphi\colon [0,1]^{\bN^m} \to [0,1]$.
  The statement follows.
\end{proof}

\begin{cor}
  \label{cor:DistributionSequentialCompactness}
  Every sequence $(\vec \mu_n)_n \subseteq \fD_{[0,1]^m}(\sA)$ admits a sub-sequence which converges weakly.
\end{cor}
\begin{proof}
  First of all, we may assume that $\sA$ is separable, since we may replace it with $\sigma\Bigl( \bigl\{ \vec \mu_n\bigl( \prod_{i<m} [0,q_i] \bigr) \bigr\}_{n\in\bN,\bar q \in \bQ^m} \Bigr)$.
  Then $\tS_m(\sA)$ is compact \emph{and} admits a countable basis, so every sequence there admits a converging sub-sequence.
\end{proof}

In case we wish to consider distributions of $\bR$-valued random variables we need to be a little more careful.

\begin{dfn}
  \label{dfn:Tight}
  A family of distributions $\fC \subseteq \fD_{\bR^m}(\sA)$ is \emph{tight} if for every $\varepsilon > 0$ there is $R \in \bR$ such that $\| \vec \mu( [-R,R]^m ) \|_1 > 1-\varepsilon$ for all $\vec \mu \in \fC$.

  We say that a family of $m$-tuples of random variables is \emph{bounded in measure} if their respective joint distributions form a tight family.
\end{dfn}

\begin{rmk}
  In \cite{BenYaacov:SchroedingersCat}, the first author pointed out (in a somewhat different formalism) that given any ``modulus of tightness'', the family of real-valued random variables respecting this modulus is interpretable as an imaginary sort in $ARV$ (or $APr$).
\end{rmk}

Let $\rho\colon [-\infty,\infty] \to [0,1]$ be any Borel map.
For $\vec \mu \in \fD_{\bR^m}(\sA)$, we may view $\vec \mu$ as a member of $\fD_{[-\infty,\infty]^m}$ and then let $\rho_*\vec \mu \in \fD_{[0,1]^m}(\sA)$ denote the image measure under $\rho$, i.e., $\rho_* \vec \mu (B) = \vec\mu \bigl( (\rho \times \cdots \times \rho)^{-1}[B] \bigr)$.

\begin{lem}
  \label{lem:TightConvergentSequence}
  Let $(\vec \mu_n)_n \subseteq \fD_{\bR^m}(\sA)$ be any sequence, and let $\rho\colon [-\infty,\infty] \to [0,1]$ be a homeomorphism.
  Then $(\vec \mu_n)_n$ converges weakly in $\fD_{\bR^m}(\sA)$ if and only if it is tight and $(\rho_*\vec \mu_n)_n$ converges weakly in $\fD_{[0,1]^m}(\sA)$.
\end{lem}
\begin{proof}
  For $R  > 0$, let $\chi_R \colon \bR^m \to [0,1]$ be continuous with $\mathbf{1}_{[-R,R]^m} \leq \chi_R \leq \mathbf{1}_{[-R-1,R+1]^m}$.
  Notice that the sequence is tight if and only if, for every $\varepsilon > 0$ there is an $R$ such that $\|\int \chi_R\, d\vec \mu_n\|_1 > 1-\varepsilon$ for all $n$.

  For left to right, assume that $\vec \mu_n \to \vec \mu$ weakly.
  Then $\rho_*\vec \mu_n \to \rho_*\vec \mu$ weakly (since there are fewer test functions).
  In addition, for each $\varepsilon > 0$ there exists $R_0$ such that $\|\int \chi_{R_0}\, d\vec \mu\|_1 > 1-\varepsilon$.
  By assumption $\| \int \chi_{R_0}\, d\vec \mu_n \|_1 \to \|\int \chi_{R_0}\, d\vec \mu\|_1$, so for some $n_0$ we have $\| \int \chi_{R_0}\, d\vec \mu_n \|_1 > 1-\varepsilon$ for all $n \geq n_0$.
  We can then find $R_1$ such that $\| \int \chi_{R_1}\, d\vec \mu_n \|_1 > 1-\varepsilon$ for all $n < n_0$.
  Let $R = \max(R_0,R_1)$.
  Then $\| \int \chi_R\, d\vec \mu_n \|_1 > 1-\varepsilon$ for all $n$ and the sequence is tight.

  For right to left, we assume that the sequence is tight and that $\rho_*\vec \mu_n \to \vec \nu$ weakly in $\fD_{[0,1]^m}(\sA)$.
  Then there exists $\vec \mu \in \fD_{[-\infty,\infty]^m}(\sA)$ such that $\rho_*\vec \mu = \vec \nu$ and $\vec \mu_n \to \vec \mu$ weakly in $\fD_{[-\infty,\infty]^m}(\sA)$.
  By tightness, for each $\varepsilon > 0$ there is $R$ such that $\| \int \chi_R\, d\vec \mu_n \|_1 > 1-\varepsilon$ for all $n$.
  By weak convergence we obtain $\| \int \chi_R \, d\vec \mu \|_1 \geq 1-\varepsilon$.
  We conclude that $\vec \mu( \bR^m) = 1$, i.e., $\mu \in \fD_{\bR^m}(\sA)$, as desired.
\end{proof}

\begin{cor}[{\cite[Theorem~1.7]{Berkes-Rosenthal:AlmostExchangeableSequences}}]
  \label{cor:TightConvergentSubSequence}
  Every tight sequence in $\fD_{\bR^m}(\sA)$ has a weakly converging sub-sequence.
\end{cor}

Next, we wish to relate the topology of strong convergence of conditional distributions with a topology on the corresponding space of types.
As a first approximation, we prove:

\begin{thm}
  \label{thm:StrongConvergenceIncomplete}
  Let $A$ be a set of parameters, and identify $\tS_m(A)$ with $\fD_{[0,1]^m}(\sigma(A))$ as above.
  Then the topology of $d$-convergence (of types) refines that of strong convergence (of distributions), which in turn refines that of $\Cb$-convergence (of types).
\end{thm}
\begin{proof}
  Let us first show that $d$-convergence implies strong convergence.
  Let $M$ be a large model containing $A$, and let $\alpha \in \bN^m$.
  Then the map $M^m \to L^1(\sigma(A))$, $\bar X \mapsto E[\bar X^\alpha|\sigma(A)]$ is continuous, where both spaces are equipped with the usual $L^1$ metric.
  It follows that the map $(\tS_m(A),d) \to L^1(\sigma(A))$, $p \mapsto \int \bar x^\alpha\,dp(\bar x)$ is continuous.

  We now prove that strong convergence implies $\Cb$-convergence.
  For this purpose we need to show that for every formula $\varphi(\bar x,\bar y)$, the map that associates $p \mapsto \Cb_\varphi(p)$ is continuous when equipping $\tS_m(A)$ with the topology of strong convergence.
  By \autoref{cor:MomentQE}, it is enough to show this where $\varphi(\bar x,\bar y) = E[\bar x^\alpha\bar y^\beta]$.
  Indeed, let $\bar X, \bar X' \in M^m$, $p = \tp(\bar X/A)$, $p' = \tp(\bar X'/A)$, $f = \bE[\bar X^\alpha|\sigma(A)]$, $f' = \bE[\bar X'^\alpha|\sigma(A)]$.
  Then for each $\bar Y \in \dcl(A)^k$ we have
  \begin{gather*}
    \bigl| \varphi(\bar X,\bar Y) - \varphi(\bar X',\bar Y) \bigr|
    = \bigr| \bE[ Y^\beta (f-f') ] \bigr|
    \leq \|f-f'\|_1,
  \end{gather*}
  so $d\bigl( \Cb_\varphi(p), \Cb_\varphi(p') \bigr) \leq \|f-f'\|_1$, and $p \mapsto \Cb_\varphi(p)$ is continuous in strong convergence.
\end{proof}

\section{Strongly finitely based (SFB) theories and lovely pairs}
\label{sec:SFB}

In order to show that the three topologies referred to in \autoref{thm:StrongConvergenceIncomplete} agree, we need to show that the canonical base topology agrees with the distance on $\tS_m(A)$.

\begin{dfn}
  \label{dfn:SFB}
  We say that a theory $T$ is \emph{strongly finitely based (SFB)} if for every model $M \vDash T$ and every $n$, the topologies $\sT_\Cb$ and $\sT_d$ agree on $\tS_m(M)$ (this does not change if we allow any algebraically closed set $A$ instead of $M$).
\end{dfn}

We recall from \autoref{sec:Convergence} that the canonical base of a type $p \in \tS_m(M)$ lies in an infinitary imaginary sort $S_{\Cb_m} = \prod_{\varphi\in \Phi_m} S_{\Cb_\varphi}$ of $M$, where $\Phi_m$ is some sufficient set of formulae as per \autoref{ntn:Phi}.
Let $\cC_m(M) \subseteq S_{\Cb_m}^M$ consist of those tuples which actually arise as canonical bases of types over $M$.
It is not difficult to see that $\cC_m(M)$ is a type-definable set, and uniformly so in all models of $T$ (see \cite[Lemma~1.3]{BenYaacov:UniformCanonicalBases}).
Since the type can be recovered from its canonical base, the map $\Cb_{M,m}$ is injective, and by definition the canonical base map $\Cb\colon (\tS_m(M),\sT_\Cb) \to \cC_m(M)$ is a homeomorphism.

\begin{prp}
  \label{prp:SFBoStab}
  Assume that $T$ is SFB.
  Then $T$ is $\aleph_0$-stable.
\end{prp}
\begin{proof}
  Let $M$ be a separable model and let $m \in \bN$.
  Since $\Phi_m$ is countable, $S_{\Cb_m}^M$ is separable, and so is its subset $\cC_m(M)$.
  Therefore $\sT_\Cb$ is separable on $\tS_m(M)$, and by SFB, $(\tS_m(M),d)$ is separable.
\end{proof}

Our next goal is to give a general criterion for SFB.
For this, let us recall a few facts regarding definable sets in continuous logic.

\begin{dfn}
  Let $M$ be any structure, $X \subseteq M$ a possibly large subset, $A \subseteq M$ a set of parameters.
  We say that $X$ is \emph{($A$-)definable} in $M$ if it is closed and the predicate $d(x,X)$ is definable (over $A$).
\end{dfn}
Definable subsets of $M^n$ are defined similarly.

Let us also recall the following result, due to the third author.
For a proof see \cite{BenYaacov-Usvyatsov:dFiniteness}.
\begin{fct}[Ryll-Nardzewski Theorem for metric structures]
  Let $T$ be a theory in a countable language.
  Then the following are equivalent:
  \begin{enumerate}
  \item $T$ is $\aleph_0$-categorical, i.e., admits a unique separable model up to isomorphism.
  \item $T$ is complete and for each $m \in \bN$, the metric topology and the logic topology on $\tS_m(T)$ agree.
  \item $T$ is complete and the metric topology and the logic topology on $\tS_\omega(T)$ agree.
  \end{enumerate}
  In particular, in an $\aleph_0$-categorical theory, every type-definable set $X$ is definable (since the map $\tp(x) \mapsto d(x,X)$ is metrically continuous, and therefore continuous, so the predicate $d(x,X)$ is definable).
\end{fct}
(Notice that the separable models include any possible compact model of $T$, so $\aleph_0$-categoricity implies completeness by Vaught's Test.)

In particular, if $T$ is $\aleph_0$-categorical, then $\cC_m$ is a definable set, i.e., $\cC_m(M)$ is uniformly definable is all models of $T$.
Definability of sets is most often used as follows:

\begin{fct}[{\cite{BenYaacov:DefinabilityOfGroups} or \cite{BenYaacov-Berenstein-Henson-Usvyatsov:NewtonMS}}]
  \label{fct:DefSet}
  Let $M$ be a structure, $X \subseteq M$ a closed, possibly large subset, $A \subseteq M$ a set of parameters.
  Then the following are equivalent:
  \begin{enumerate}
  \item The set $X$ is $A$-definable.
  \item For every formula $\varphi(x,\bar y)$ (possibly over $A$), the predicate $\psi(\bar y) = \inf_{x\in X} \varphi(x,\bar y)$ is definable by a formula over $A$ as well.
  \end{enumerate}
\end{fct}

Let us now recall a few facts regarding Poizat's beautiful pairs \cite{Poizat:Paires}.
We define an \emph{elementary pair} of models of $T$ to be a pair $(M,N)$, where $N \prec M \vDash T$.
We view such a pair as a structure $(M,P)$ in $\cL_P = \cL \cup \{P\}$, where $P$ is a new $1$-Lipschitz unary predicate symbol measuring the distance to $N$, and we may also write $N = P(M)$.
A \emph{beautiful pair} of models of $T$ is an elementary pair $(M,P)$ such that  $P(M)$ is $|\cL|^+$-saturated, and $M$ is $\aleph_0$-saturated over $P(M)$.
We define $T_P$ as the $\cL_P$-theory of all beautiful pairs of models of $T$.
If saturated models of $T_P$ are not beautiful pairs (which may happen, for example, if $T$ is a classical stable theory with the finite cover property) then (continuous) first order logic is not adequate for the consideration of the class of beautiful pairs.
(On the other hand, \emph{positive logic} always provides an adequate framework, see \cite{BenYaacov:NonFirstOrderLovelyPairs}.)
If saturated models of $T_P$ \emph{are} beautiful pairs then continuous first order logic is adequate and we shall say that the class of beautiful pairs of models of $T$ is \emph{almost elementary}.

\begin{fct}
  \label{fct:BeautifulPairs}
  Assume that $T$ is $\aleph_0$-categorical, or more generally, that $\cC_m$ is a definable set for all $m$.
  Then the class of beautiful pairs of models of $T$ is almost elementary.
\end{fct}
\begin{proof}
  See \cite[Theorem~4.4]{BenYaacov:UniformCanonicalBases}.
  To sketch the argument, one can always express that $(M,P)$ is an elementary pair.
  Since $\cC_m$ is definable, one can quantify over it and express that for every $p \in \tS_m(M)$ (i.e., for every canonical base of such type), and every finite subset $A \subseteq M$, sufficiently good approximations (uniform, in finitely many formulae) of the restriction $p\rest_{A \cup P}$ are realised in $M$.
  This is true in every beautiful pair, and conversely, if $(M,P)$ is sufficiently saturated and satisfies this theory then $p\rest_{A \cup P}$ is actually realised, so $(M,P)$ is beautiful.
\end{proof}

\begin{lem}
  \label{lem:PCbDef}
  Let $(M,P)$ be an elementary pair of models and $\varphi(\bar x,\bar y) \in \Phi_m$.
  Then the map $\bar a \mapsto \Cb_\varphi(\bar a/P)$ is uniformly definable in $(M,P)$, i.e., its graph is definable by a partial type which does not depend on $(M,P)$.
\end{lem}
\begin{proof}
  The graph of $z = \Cb_\varphi(\bar x/P)$ is defined by:
  \begin{gather*}
    P(z) = 0 \quad \& \quad \sup_{\bar y \in P} \, \bigl| \varphi(\bar x,\bar y)-d_{\bar x}\varphi(\bar y,z) \bigr| = 0.
  \end{gather*}
  (See also \cite{BenYaacov:UniformCanonicalBases}.)
\end{proof}

It follows that for every $m$ we have a uniformly definable map $\theta\colon (M,P)^m \to \cC_m(P)$ inducing a continuous function $\hat \theta\colon \tS_m(T_P) \to \tS_{\Cb_m}(T)$ given as follows (here $\tS_{\Cb_m}(T)$ is the space of types in the sort $S_{\Cb_m}$).
\begin{gather}
  \label{eq:Theta}
  \begin{array}{cccc}
    \theta\colon & \bar a & \mapsto & \Cb(\bar a/P), \\
    & \\
    \hat \theta\colon & \tp^{\cL_P}(\bar a) & \mapsto &
    \tp\bigl( \Cb(\bar a/P) \bigr) = \tp\bigl( \theta(\bar a) \bigr).
  \end{array}
\end{gather}

\begin{fct}
  \label{fct:BPTypes}
  Let $(M,P)$ and $(N,P)$ be two beautiful pairs of models of $T$ and let $\bar a \in M$ and $\bar b \in N$ be two $m$-tuples.
  Then $\bar a \equiv^{\cL_P} \bar b$ if and only if $\theta(\bar a) \equiv \theta(\bar b)$, i.e., if and only if $\hat \theta(\bar a) = \hat \theta(\bar b)$.
\end{fct}
\begin{proof}
  One direction holds since $\hat \theta$ is well defined.
  The converse is proved as for \cite[Théorème~4]{Poizat:Paires}, checking that the family of finite partial maps $f\colon M \rightarrow N$ such that $\theta(\dom f) \equiv \theta(\img f)$ forms a back-and-forth system between $(M,P)$ and $(N,P)$.
\end{proof}

\begin{prp}
  \label{prp:TPTypes}
  Assume the class of beautiful pairs of models of $T$ is almost elementary.
  Then the map $\hat \theta$ defined above is a homeomorphic embedding.
\end{prp}
\begin{proof}
  We have already observed that $\hat \theta$ is a continuous map from a compact space into a Hausdorff space.
  Therefore, all we need to show is that it is injective.
  Let $(M,P), (N,P) \vDash T_P$, $\bar a \in M^m$, $\bar b \in N^m$, and assume that $\hat \theta(\bar a) = \hat \theta(\bar b)$.
  We may replace both $(M,P)$ and $(N,P)$ by $|\cL|^+$-saturated elementary extensions.
  By assumption $(M,P)$ and $(N,P)$ are beautiful pairs and we may apply \autoref{fct:BPTypes}.
\end{proof}

\begin{thm}
  \label{thm:SFBEqOCatPairs}
  Let $T$ be any stable continuous first order theory.
  Then $T_P$ is $\aleph_0$-categorical if and only if $T$ is $\aleph_0$-categorical and SFB.
\end{thm}
\begin{proof}
  Assume first that $T_P$ is $\aleph_0$-categorical.
  Then clearly $T$ is $\aleph_0$-categorical (indeed, if $\tS_m(T_P)$ is metrically compact then so is $\tS_m(T)$).

  So fix $m \in \bN$ and let $\bar x$ be an $m$-tuple.
  We shall in fact prove a uniform version of SFB, namely that for every $\varepsilon > 0$ there is $\delta > 0$
  such that if $N \vDash T$ and $p,q \in \tS_m(N)$ are such that $d(\Cb(p),\Cb(q)) < \delta$ (where the distance between canonical bases is as defined in \autoref{eq:CbMet}), then $d(p,q) \leq \varepsilon$.
  For simplicity of notation we shall assume that $m = 1$ and drop the bars.

  Recall from \autoref{lem:PCbDef} that the map $\theta\colon a \mapsto \Cb(a/P)$ is uniformly definable in $T_P$.
  Let $r(x,y)$ be the partial $\cL_P$-type saying that $x \equiv_P y$.
  Since $T_P$ is $\aleph_0$-categorical, the distance $d(xy,r)$ is a definable predicate.
  Consider now the partial $\cL_P$-type consisting of $\{d(xy,r) \geq \varepsilon/2\} \cup \{d\bigl( \theta(x),\theta(y) \bigr) < \delta\}_{\delta > 0}$.
  This partial type is contradictory, whence we obtain a $\delta > 0$ such that
  \begin{gather*}
    d\bigl(\theta(x),\theta(y)\bigr) < \delta
    \vdash
    d(xy,r) < \varepsilon/2.
  \end{gather*}
  We claim that this $\delta$ is as required, i.e., if $N \vDash T$, $p,q \in \tS_1(N)$, and $d\bigl( \Cb(p),\Cb(q) \bigr) < \delta$, then $d(p,q) \leq \varepsilon$.
  Indeed, passing to an elementary extension and taking non forking extensions of the types we may assume that $N$ is $\aleph_1$-saturated, and then find $M \succ N$ which is $|N|^+$-saturated, so $(M,N) = (M,P) \vDash T_P$ is a beautiful pair.
  Let $C = \Cb(p)$, $D = \Cb(q)$, so $d(C,D) < \delta$.

  By our saturation assumption there exist $a,b \in M$ such that $a \vDash p$ and $b \vDash q$, so $\theta(a) = C$, $\theta(b) = D$, and therefore $d( ab,r ) < \varepsilon/2$.
  In other words, there exist $a'b' \in M$ such that $d(ab,a'b') < \varepsilon/2$ and $\tp(a'/N) = \tp(b'/N) = p'$, say.
  Therefore
  \begin{gather*}
    d(p,q) \leq d(p,p') + d(p',q)
    < \varepsilon/2+\varepsilon/2 = \varepsilon.
  \end{gather*}

  Conversely, assume that $T$ is $\aleph_0$-categorical and is SFB.
  By the metric Ryll-Nardzewsky Theorem, we need to show that for each $m$, the logic topology and the metric topology on $\tS_m(T_P)$ coincide.
  In other words, we need to show that if $p_n \to p$ in $\tS_m(T_P)$, then $p_n \to^d p$ there.

  Assume then that $p_n \to p$.
  Since $T$ is $\aleph_0$-categorical, the class of beautiful pairs of models of $T$ is almost elementary (\autoref{fct:BeautifulPairs}).
  By \autoref{prp:TPTypes}, the map $\hat \theta \colon \tp^{\cL_P}(a) \mapsto \tp\bigl( \Cb(a/P) \bigr)$ is a topological embedding, so $\hat \theta(p_n) \to \hat \theta(p)$, and since $T$ is $\aleph_0$-categorical we have $\hat \theta(p_n) \to^d \hat \theta(p)$.
  In other words, in a sufficiently saturated model $N \vDash T$ we can find infinite tuples $C_n \vDash \hat \theta(p_n)$ and $C \vDash \hat \theta(p)$ such that $C_n \to C$.

  Write $C = \{c_\varphi\}_{\varphi\in \Phi_m}$, and let $q \in \tS_m(N)$ be the unique type over $N$ such that $\Cb(q) = C$, i.e.,
  \begin{gather*}
    \varphi(x,b)^q = d_x\varphi(b,c_\varphi),
    \qquad
    b \in N,
    \varphi \in \Phi(x).
  \end{gather*}
  Define $q_n \in \tS_m(N)$ such that $\Cb(q_n) = C_n$ similarly.
  Then $q_n \to^\Cb q$ by definition, and since $T$ is SFB $q_n \to^d q$.
  Let $a_n \vDash q_n$ and $a \vDash q$ witness this, so $a_n \to a$ in some $M \succeq N$, which we may assume to be $|N|^+$-saturated, so $(M,N) = (M,P)$ is a beautiful pair.
  Then $\theta(a) = C \vDash \hat \theta(p)$ implies $p = \tp^{\cL_P}(a)$, and similarly $\tp^{\cL_P}(a_n) = p_n$.
  Thus $a_n \to a$ witnesses that $p_n \to p$, and the proof is complete.
\end{proof}

The intuitive idea behind this criterion is roughly as follows.
We assume that $T$ is $\aleph_0$-categorical, and let $(M,N)$ be a lovely pair of models thereof.
Then every type in $\tS_m(N)$ is realised by some $\bar a \in M^m$, and the map $\tp(\bar a/N) \mapsto \tp^{(M,N)}(\bar a)$ is a well defined surjection $\tS_m(N) \rightarrow \tS_m(T_P)$.
Since formulae in $T_P$ essentially give information about $\Cb(\bar x/P)$ (compare with the more explicit approach of \cite[Section~4]{BenYaacov:UniformCanonicalBases}), and since $T$ is assumed to be $\aleph_0$-categorical, the logic topology on $\tS_m(T_P)$ agrees with the quotient of the canonical base topology on $\tS_m(N)$.
On the other hand, the distance topology on $\tS_m(T)$ is the quotient of the distance topology on $\tS_m(N)$.
Thus, the gap between the logic and distance topologies on $\tS_m(T_P)$ (i.e., $T_P$ being $\aleph_0$-categorical or not) boils down, more or less, to the gap between the canonical base and distance topologies over a model of $T$ (i.e., $T$ being SFB or not).

In the case of classical (discrete) first order logic, the situation covered by \autoref{thm:SFBEqOCatPairs} boils down to the one covered by the following result of Zilber et al.

\begin{fct}[{\cite[Theorem~5.12]{Pillay:GeometricStability}}]
  \label{fct:Zilber}
  An $\aleph_0$-categorical, $\aleph_0$-stable classical theory is one-based.
\end{fct}

\begin{prp}
  \label{prp:ZilberEquivs}
  Let $T$ be a classical $\aleph_0$-categorical (and stable) theory.
  Then the following are equivalent:
  \begin{enumerate}
  \item $T$ is SFB.
  \item $T$ is $\aleph_0$-stable.
  \item $T$ is one based.
  \item $T$ is finitely based (meaning that for every $m$ there exists $k$ such that every indiscernible sequence of $m$-tuples, is a Morley sequence over its first $k$ elements).
  \end{enumerate}
\end{prp}
\begin{proof}
  \begin{cycprf}
  \item[\impnext]
    We have already seen that SFB implies $\aleph_0$-stability.
  \item[\impnext]
    By \autoref{fct:Zilber}.
  \item[\impnext]
    Immediate ($k = 1$).
  \item[\impfirst]
    Let us fix $m = 1$ and the corresponding $k$.
    Then the type of an indiscernible sequence (of singletons) is determined by the type of the first $k+1$ members of that sequence, so only finitely many types of indiscernible sequences exist.
    On the other hand, if $(M,P) \vDash T_P$ and $a \in M$, then $\tp^{\cL_P}(a)$ is determined by the $\cL$-type of $\Cb(a/P)$, which in turn is determined by the type of a Morley sequence in $\tp(a/P)$.
    We conclude that $\tS_1(T_P)$ is finite, and by similar reasoning so is $\tS_m(T_P)$ for all $m$.
    Therefore $T_P$ is $\aleph_0$-categorical, so $T$ is SFB by \autoref{thm:SFBEqOCatPairs}.
  \end{cycprf}
\end{proof}

Our \autoref{thm:SFBEqOCatPairs} is therefore mostly interesting for $\aleph_0$-categorical continuous theories, to which \autoref{prp:ZilberEquivs} does \emph{not} generalise.
For the direction ``one-based $\Longrightarrow$ SFB'' we merely observe that the proof given above does not carry over to the metric setting.
For the direction ``SFB $\Longrightarrow$ one-based'' we present below a counter-example.

\begin{ntn}
  For any theory $T$, let $T_{P,0}$ denote the theory of elementary pairs of models of $T$ in the language $\cL_P$ (which is an elementary class).
\end{ntn}

\begin{cor}
  \label{cor:HilbertSFB}
  The theory of infinite dimensional Hilbert spaces $IHS$ is SFB.
\end{cor}
\begin{proof}
  Let $IHS'_P$ consist of $IHS_{P,0}$ together with the axiom scheme expressing, for each $k$, that there exist $k$ orthonormal vectors which are orthogonal to $P$ (we leave the details to the reader, pointing out that since $P$ is definable modulo $IHS_P$, one may quantify over it).
  It is then not difficult to check that every beautiful pair of models of $IHS$ is a model of $IHS'_P$, so $IHS'_P \subseteq IHS_P$.
  On the other hand, $IHS'_P$ admits a unique separable model $(H \oplus H_1, H)$ where $H \cong H_1 \vDash IHS$ are separable.
  Thus $IHS'_P$ is complete, so $IHS'_P = IHS_P$, and $IHS$ is SFB by \autoref{thm:SFBEqOCatPairs}.
\end{proof}

\begin{cor}
  \label{cor:ProbAlgSFB}
  The theories $APr$ and $ARV$ are SFB.
\end{cor}
\begin{proof}
  The argument is essentially the same as above.
  We define $APr'_P$ to consist of $APr_{P,0}$ along with the axiom saying that $M$ is atomless over $P$, expressible as
  \begin{gather*}
    \sup_x \, \inf_y \, \sup_{z \in P} \, \left| \half \mu(x \cap z) - \mu(x \cap y \cap z) \right| = 0.
  \end{gather*}
  Then again every beautiful pair is a model of $APr'_P$ and $APr'_P$ admits a unique separable model, namely $(\fB(X\times Y),\fB(X))$ where $X = Y = [0,1]$ is equipped with the Lebesgue measure, and the embedding $\fB(X) \hookrightarrow \fB(X\times Y)$ is induced by the projection $X\times Y \twoheadrightarrow X$.

  Since $ARV$ and $APr$ are biïnterpretable, SFB follows for $ARV$.
  Alternatively, the same argument holds for $ARV$, where atomlessness of $\sigma(M)$ over $\sigma(P)$ is expressed by
  \begin{gather*}
    \sup_x \, \inf_y \, \sup_{z \in P} \, \left| \half E(x \wedge z) - E(x \wedge y \wedge z) \right| = 0.
  \end{gather*}
\end{proof}

\begin{thm}
  \label{thm:StrongConvergence}
  Let $A$ be a set of parameters, and identify $\tS_n(A)$ with $\fD_n(\sigma(A))$ as above.
  Then the topologies of $d$-convergence, $\Cb$-convergence (of types) and strong convergence (of distributions) agree.
\end{thm}
\begin{proof}
  By \autoref{thm:StrongConvergenceIncomplete} and \autoref{cor:ProbAlgSFB}.
\end{proof}

Both of the examples above are $\aleph_0$-categorical and $\aleph_0$-stable, so it is natural to expect them to satisfy some continuous analogue of one-basedness.
It is not difficult to verify that none of them is literally one-based.
In fact, no known continuous stable theory is one based, except for those constructed trivially from classical ones.
Given the examples above, and in analogy with \autoref{prp:ZilberEquivs}, it stands to reason to contend that at least for $\aleph_0$-categorical theories, SFB is the correct continuous logic analogue of a classical one-based theory, and one may further formalise it as a conjecture:

\begin{conj}[Zilber's Theorem for continuous logic, naïve version]
  Every $\aleph_0$-categorical $\aleph_0$-stable theory is SFB.
\end{conj}

Unfortunately, this conjecture has an easy counterexample:

\begin{exm}
  The theory $ALpL$ of atomless $L^p$ Banach lattices for $p \in [1,\infty)$ (see \cite{BenYaacov-Berenstein-Henson:LpBanachLattices}) is not SFB.
  This has already been observed in \autoref{exm:LpNotSFB} using results of \cite{BenYaacov:UniformCanonicalBases}.
  This can also be observed using our criterion, as follows.

  A model of $ALpL_P$ is of the form $\bigl( L^p(X,\fB_X,\mu_X), L^p(Y,\fB_Y,\mu_Y) \bigr)$, where $\fB_Y \subseteq \fB_X$ (so in particular $Y \subseteq X$) and $\mu_Y = \mu_X\rest_{\fB_Y}$, such that in addition $\mu_Y$ is atomless and $\mu_X$ is atomless over $\fB_Y$.
  The theory $ALpL_P$ has precisely two non isomorphic separable models, one where $Y = X$ and the other where $\mu(X \setminus Y) > 0$.

  We may construct them explicitly as $\bigl( L^p(X\times Y),L^p(X) \bigr)$ and $\bigl( L^p(Z\times Y),L^p(X) \bigr)$, where $X=Y = [0,1] \subseteq Z = [0,2]$ are equipped with the Lebesgue measure, the embedding $L^p(X) \subseteq L^p(X \times Y)$ is given by $f'(x,y) = f(x)$ and $L^p(X \times Y) \subseteq L^p(Z \times Y)$ is given by $f'(w) = f(w)$ for $w \in X \times Y$, $f'(w) = 0$ otherwise.
\end{exm}

It is worthwhile to point out that this last example is disturbing on several other ``counts'':
\begin{itemize}
\item It is a counter-example for Vaught's no-two-models theorem in continuous logic.
\item Since $ALpL$ is $\aleph_0$-stable, $ALpL_P$ is superstable by \cite{BenYaacov:SuperSimpleLovelyPairs}, and we get a counter-example to Lachlan's theorem on the number of countable models of a first order superstable theory.
\end{itemize}

Nonetheless, one may still hope to recover a version of Zilber's Theorem for continuous logic using the notion of perturbations of metric structures (as introduced in \cite{BenYaacov:Perturbations,BenYaacov:TopometricSpacesAndPerturbations}).
Natural considerations suggest that whenever adding symbols to a language (especially to the language of an $\aleph_0$-categorical theory) one should also study the expanded structures up to arbitrarily small perturbations of the new symbol.
Thus, the question should not be whether $ALpL_P$ is $\aleph_0$-categorical, but rather, whether it is $\aleph_0$-categorical up to small perturbations of the predicate $P$
(the positive non-perturbed results for $IHS_P$ and $APr_P$ should be viewed witnessing the exceptional structural simplicity of these theories).

\begin{prp}
  The theory $ALpL_P$ is $\aleph_0$-categorical up to arbitrarily small perturbations of $P$.
\end{prp}
\begin{proof}
  We need to show that if $(M,P),(N,P) \vDash ALpL_P$ are separable then there exists an isomorphism $\rho\colon M \to N$ such that $|d(f,P) - d(\rho(f),P)| < \varepsilon$ for all $f \in M$.
  Since $ALpL_P$ has precisely two non-isomorphic separable models, it will suffice to show this for those two models.

  Let $N$, $M_1$ and $M_2$ be the closed unit balls of $L^p([0,1])$, $L^p([0,1] \times [0,1])$ and $L^p([0,2]\times[0,1])$, respectively (with the Lebesgue measure).
  As in the example above, we consider that $N \subseteq M_1 \subseteq M_2$.
  In particular, $N$ is the set of all $g \in M_1$ such that the value of $g(x,y)$ depends only on $x$.
  Then the two non-isomorphic models are $(M_1,N)$ and $(M_2,N)$.

  Define $\rho_1\colon L^p([0,1]\times[0,1]) \to L^p([0,1]\times[\varepsilon,1])$ and $\rho_2\colon L^p([1,2]\times[0,1]) \to L^p([0,1]\times[0,\varepsilon])$ by:
  \begin{align*}
    (\rho_1f)(x,y) & = (1-\varepsilon)^{-1/p} f\bigl( x,(y-\varepsilon)/(1-\varepsilon) \bigr)
    \\
    (\rho_2f)(x,y) & = \varepsilon^{-1/p}f(x+1,y/\varepsilon).
  \end{align*}
  Then $\rho_1$ and $\rho_2$ are isomorphisms of Banach lattices, which can be combined into an isomorphism $\rho = \rho_1\oplus \rho_2 \colon L^p([0,2]\times[0,1]) \to L^p([0,1]\times[0,1])$.
  This restricts to an isomorphism of the unit balls which will also be denoted by $\rho\colon M_2 \to M_1$.
  Let also $D = [0,1] \times [0,\varepsilon]$ and $E = [0,1] \times [\varepsilon,1]$, namely the supports of the images of $\rho_2$ and $\rho_1$, respectively.

  We claim that $\rho\rest_N\colon N \to M_1$ is not too far from the identity.
  Indeed, let $g \in N$.
  Then $\|g\| \leq 1$, and we can write it as a function of the first coordinate $g(x)$.
  Then $\rho(g) = \rho_1(g)$ can be written as $(1-\varepsilon)^{-1/p}g(x)\chi_E(x,y)$.
  For $r \in [0,1]$ let:
  \begin{gather*}
    \zeta(r) = 1 - (1-r)^{1/p} + r^{1/p}.
  \end{gather*}
  Then:
  \begin{align*}
    \|g-\rho(g)\|
    &
    \leq \|g\chi_E - (1-\varepsilon)^{-1/p}g\chi_E\| + \|g\chi_D\|
    \\ &
    = \left((1-\varepsilon)^{-1/p}-1\right)\|g\chi_E\| + \|g\|\varepsilon^{1/p}
    \\ &
    = \|g\|\left((1-\varepsilon)^{-1/p}-1\right)(1-\varepsilon)^{1/p} + \|g\|\varepsilon^{1/p}
    \\ &
    = \|g\|\zeta(\varepsilon) \leq \zeta(\varepsilon).
  \end{align*}
  Now let $f \in M_2$.
  Then:
  \begin{align*}
    \bigl| \|f-g\| - \|\rho(f)-g\| \bigr|
    &
    = \bigl| \|\rho(f)-\rho(g)\| - \|\rho(f)-g\| \bigr|
    \\ &
    \leq \|g-\rho(g)\| \leq \zeta(\varepsilon).
  \end{align*}
  Fixing $f \in M_2$ while letting $g \in N$ vary, we conclude that:
  \begin{align*}
    \bigl| d(\rho f,P)^{(M_1,N)} - d(f,P)^{(M_2,N)} \bigr|
    &
    = \bigl| d(f,N) - d(\rho(f),N) \bigr|
    \\ &
    \leq \sup_{g\in N} \bigl| \|f-g\| - \|\rho(f)-g\| \bigr|
    \leq \zeta(\varepsilon).
  \end{align*}
  Since $\zeta$ is continuous and $\zeta(0) = 0$, by taking $\varepsilon > 0$ small enough we can get $\rho\colon (M_2,N) \to (M_1,N)$ to be as small a perturbation of the predicate $P(x) = d(x,N)$ as we wish.
\end{proof}

We therefore propose the following:

\begin{conj}[Zilber's Theorem for continuous logic]
  Whenever $T$ is an $\aleph_0$-categorical $\aleph_0$-stable theory (in a countable language) $T_P$ is $\aleph_0$-categorical up to arbitrarily small perturbations of the predicate $P$.
\end{conj}

\section{Almost indiscernible sequences and sub-sequences}
\label{sec:AlmostIndiscernible}

One of the questions studied by Berkes \& Rosenthal \cite{Berkes-Rosenthal:AlmostExchangeableSequences} is when a sequence of random variables possesses an almost exchangeable sub-sequence.
In this section we address the corresponding model-theoretic question, namely, when a sequence of tuples possesses an almost indiscernible sub-sequence.
For simplicity of notation, we only consider (sequences of) singletons, but the everything we prove holds just as well for arbitrary tuples.

\begin{dfn}
  \label{dfn:AlmostIndiscernible}
  A sequence $(a_n)_{n\in\bN}$ is \emph{almost indiscernible} if there exists (possibly in an elementary extension) an indiscernible sequence $(b_n)_{n\in\bN}$ in the same sort such that $d(a_n,b_n) \to 0$.
\end{dfn}

\begin{lem}
  \label{lem:AILim}
  Let $(a_n)_{n\in\bN}$ be an almost indiscernible sequence, say witnessed by an indiscernible sequence $(b_n)_{n\in\bN}$, and let $B \supseteq (a_n)_n$.
  Then $p = \lim \tp(a_n/B)$ exists and is stationary, and $\Cb(p) \subseteq \dcl(a_n)_n$.
  Moreover, $(b_n)_{n\in\bN}$ is a Morley sequence in $p\rest_{\Cb(p)}$.
\end{lem}
\begin{proof}
  Let $B' = B\cup\{b_n\}_n$.
  Since $T$ is stable, $r = \lim \tp(b_n/B')$ exists and is stationary, and $(b_n)_n$ is a Morley sequence in $r\rest_{\Cb(r)}$.
  By \cite{BenYaacov:StableGroups} (Lemma~4.2 and the discussion following Proposition~5.2), $\Cb(r)$ can uniformly recovered from any Morley sequence in $r$.
  Consider now an automorphism of an ambient monster model which fixes $(a_n)_n$.
  For $k$ large enough, it will move the tail $(b_n)_{n\geq k}$, and therefore $\Cb(r)$, as little as we wish.
  Therefore $\Cb(r) \subseteq \dcl(a_n)_n \subseteq \dcl(B)$.

  Clearly $\lim \tp(a_n/B) = \lim \tp(b_n/B) = r\rest_B$, so in particular the first limit exists, call it $p$.
  Since $\Cb(r) \subseteq B$, the type $p$ is stationary, $\Cb(p) = \Cb(r) \subseteq \dcl(a_n)_n$.
  Finally, $r\rest_{\Cb(r)} = p\rest_{\Cb(p)}$.
\end{proof}

In a discrete sort, an almost indiscernible sequence is just one which is eventually indiscernible, so having an almost indiscernible sub-sequence is the same as having an indiscernible sub-sequence.
In metric sorts, however, the two notions may differ and it is the weaker one (namely, having an almost indiscernible sub-sequence) which we shall study.

\begin{dfn}
  \label{dfn:PropertyStar}
  Let $B$ be a set containing a sequence $(a_n)_{n\in\bN}$.
  We say that $(a_n)_n$ satisfies $(*_B)$ if $p = \lim \tp(a_n/B)$ exists and is stationary, and for $C = \Cb(p)$ and $c \vDash p$ we have:
  \begin{gather*}
    \tp(Ba_n/C) \to^d \tp(Bc/C)
  \end{gather*}
  If $B = \{a_n\}_n$ we omit it and say that $(a_n)_n$ satisfies $(*)$.
\end{dfn}

Notice that property $(*_B)$ lies between convergence of $\tp(a_n/B)$ in the logic topology and convergence in $d$.
Convergence in $d$ would just mean that the sequence $(a_n)_n$ converges (in fact, canonical base convergence would imply the same, when applicable, by \autoref{rmk:CbConvergenceDistanceToModel}), and is therefore too restrictive for our purposes.
On the other hand, convergence in the logic topology alone is too weak: if $(a_n)$ is an arbitrary sequence contained, say, in a separable $B$, then by compactness there exists a sub-sequence such that $\tp(a_{n_k}/B)$ converges, so this kind of hypothesis tells us essentially nothing about the sequence -- and using stability, one may remove the separability assumption (in a countable language).
With \autoref{dfn:PropertyStar}, however, we can prove:

\begin{thm}
  \label{thm:GenBR}
  If $T$ is stable and the sequence $(a_n)_{n\in\bN} \subseteq B$ has a sub-sequence satisfying $(*_B)$ then $(a_n)_{n\in\bN}$ also has an almost indiscernible sub-sequence.
  If $T$ is superstable then the converse holds as well.

  Moreover, if in addition $q = \lim \tp(a_n/B)$ exists then it is stationary and the sequence witnessing almost indiscernibility is Morley over $\Cb(q)$.
\end{thm}
\begin{proof}
  We may assume that the sequence $(a_n)_n$ satisfies $(*_B)$, and therefore $(*)$.
  Let $A = \{a_n\}_n \subseteq B$ and let $p = \lim \tp(a_n/A)$, $C = \Cb(p)$, $c \vDash p$, so $c \ind_C A$.

  We construct by induction on $i \in \bN$ an increasing sequence $(n_i)_i$ and copies $A^ic^i$ of $Ac$, $A^i = \{a_n^i\}_n$, such that:
  \begin{enumerate}
  \item $d(a_{n_j}^i,a_{n_j}^{i+1}) \leq \frac{1}{2^i}$ for $j<i$.\vspace{1mm}
  \item $d(c^i,a_{n_i}^{i+1}) \leq \frac{1}{2^i}$.
  \item $A^ic^i \equiv_C Ac$.
  \item $c^i \ind_C c^{<i}$.
  \end{enumerate}
  We start with $A^0c^0 = Ac$.
  At the $i$th step we already have $A^i$, $c^i$, and $n_{<i}$.
  By $(*)$ there exists $k$ such that:
  \begin{gather*}
    d\bigl( \tp(a_{n_{< i}}a_k/C),\tp(a_{n_{<i}}c/C) \bigr) \leq 2^{-i}.
  \end{gather*}
  We let $n_i = k$, and we may assume that $n_i > n_j$ for $j < i$.
  Since $Ac \equiv_C A^ic^i$, there exists $A^{i+1} \vDash \tp(A/C)$ such that
  \begin{gather*}
    d(a_{n_{\leq i}}^{i+1},a_{n_{<i}}^ic^i) \leq 2^{-i}.
  \end{gather*}

  This takes care of the first two requirements.
  Choose $c^{i+1} \vDash p\rest_C$ such that $c^{i+1} \ind_C A^{i+1}c^{\leq i}$.
  Then the two last requirements are satisfied as well, and the construction may proceed.

  For each $i$, the sequence $(a_{n_i}^j)_j$ is Cauchy, converging to a limit $b^i$, and we have $d(c^i,b^i) \leq 2^{-i+2}$.
  Also, the sequence $(c^i)_i$ is indiscernible (being a Morley sequence in $p\rest_C$), and $(b^i)_i \equiv_C (a_{n_i})_i$.
  Thus $(a_n)$ admits an almost indiscernible sub-sequence $(a_{n_i})_i$.

  For the moreover part, we may again assume that the entire sequence $(a_n)_n$ satisfies $(*_B)$, since the limit type, if it exists, must be equal to the limit type of any sub-sequence.
  Then the statement follows from \autoref{lem:AILim}.

  For the converse we assume that $T$ is superstable, and we may further assume that $(a_n)_n$ is almost indiscernible as witnessed by an indiscernible sequence $(c_n)_n$.
  By \autoref{lem:AILim}, $p = \lim \tp(a_n/B)$ exists and is stationary, and $(c_n)_n$ is a Morley sequence over $C = \Cb(p)$.
  Let $c \vDash p$, so $c \ind_C B$.

  Fix a finite tuple $\bar b \subseteq B$, and $\varepsilon > 0$.
  By superstability there exists $n \in \bN$ such that $\bar b^\varepsilon \ind_{Cc_{<n}} c_n$.
  In other words, there exists $\bar b' \equiv_{\acl(Cc_{<n})} \bar b$ such that $d(\bar b,\bar b') \leq \varepsilon$ and $\bar b' \ind_{Cc_{<n}} c_n$.
  By transitivity, $\bar b' \ind_C c_n$, in which case $\bar b'c_n \equiv_C \bar bc$.
  This proves that $\tp(Bc_n/C)\to^d\tp(Bc/C)$.
  Since $d(a_n,c_n)\to 0$, it follows that $\tp(Ba_n/C)\to^d\tp(Bc/C)$ as desired.
\end{proof}

\begin{rmk}
  \label{rmk:PropertyStar}
  Assume that $(a_n)_n$ satisfies $(*_B)$, with $c$ and $C$ as in \autoref{dfn:PropertyStar}, and let $B \supseteq B' \supseteq (a_n)_n$.
  By the proof of \autoref{thm:GenBR}, there exists a Morley sequence $(c_k)_k$ in $\tp(c/C)$ which witnesses that a sub-sequence $(a_{n_k})_k$ is almost indiscernible.
  By \autoref{lem:AILim} it follows that $C \subseteq \dcl(\{a_n\}_n)$, and therefore $(*_{B'})$ holds (so in particular $(*_B) \Longrightarrow (*)$).

  In addition, the condition $(*_B)$ is equivalent to:
  \begin{quote}
    There exists a stationary type $p \in \tS(B)$ such that if $C = \Cb(p)$ and $c \vDash p$ then $\tp(Ba_n/C) \to^d \tp(Bc/C)$.
  \end{quote}
  Indeed, this already implies that $\tp(a_n/B) \to \tp(c/B)$.
\end{rmk}

\section{The model theoretic contents of Berkes \& Rosenthal \cite{Berkes-Rosenthal:AlmostExchangeableSequences}}
\label{sec:BR}

The main motivation for the present paper is to give a formal model-theoretic account for several results of Berkes \& Rosenthal \cite{Berkes-Rosenthal:AlmostExchangeableSequences}, which have a strong model-theoretic flavour to them.
In \autoref{sec:RandomVariables} (together with \autoref{thm:StrongConvergence}) we have already related some probability-theoretic notions with model-theoretic ones (most notably, the strong and weak topologies on distributions/types).
At this stage we have the necessary tools to address the main result (Theorem~2.4) of Berkes \& Rosenthal \cite{Berkes-Rosenthal:AlmostExchangeableSequences}.

A word of caution is in place, regarding the fact that Berkes \& Rosenthal consider $\bR$-valued random variables, whereas model theory can only deal with uniformly bounded random variables (or, more generally, families bounded in measure), and the literature treats $[0,1]$-valued ones.
In \autoref{sec:RandomVariables} only topological (and not, say, algebraic) properties of $\bR$ were actually used, we could simply compose with some fixed homeomorphism $\rho\colon \bR \rightarrow (0,1)$.
The same holds in what follows with one exception, namely exchangeability, which we shall treat explicitly in \autoref{lem:IndiscExchgSubSeq2}.

The following definitions were given in \cite{Berkes-Rosenthal:AlmostExchangeableSequences} for sequences of single random variables.
We give the obvious extensions to sequences of tuples of a fixed length.
\begin{dfn}
  \label{dfn:BRPaper}
  Let $(\bar X_n)_{n\in\bN}$ be a sequence of $m$-tuples of random variables, $\bar X_n = (X_{n,0},\ldots,X_{n,m-1})$.
  \begin{enumerate}
  \item
    Let $\sC \supseteq \sigma\bigl( \{X_{n,i}\}_{n,i} \bigr)$ be any probability algebra with respect to which all the $X_{n,i}$ are measurable.
    Then the sequence is \emph{determining in $\sC$} if the sequence $\dist(\bar X_n|\sC)$ converges weakly in $\fD_{\bR^m}(\sC)$.
    (We use an alternative characterisation from \cite[Proposition~2.1]{Berkes-Rosenthal:AlmostExchangeableSequences}.)
  \item
    Let $(\bar X_n)_n$ be a determining sequence of $\sC$-measurable random variables, and let $\vec \mu \in \fD_{\bR^m}(\sC)$ be the limit distribution.
    Then the \emph{limit tail algebra} \cite[p.~474]{Berkes-Rosenthal:AlmostExchangeableSequences} of $(\bar X_n)_n$ is $\sigma(\vec \mu) = \sigma\Bigl( \bigl\{ \vec \mu\bigl( \prod (-\infty,q_i) \bigr)  \bigr\}_{\bar q \in \bQ^m} \Bigr) \subseteq \sC$.
  \item
    The sequence is \emph{exchangeable} if the joint distribution (over the trivial algebra) of any $k$ distinct tuples of the sequence depends only on $k$.
  \item
    \label{item:BRPaperAlmostExchangeaable}
    The sequence is \emph{almost exchangeable} if there is an exchangeable sequence $(\bar Y_n)_{n\in\bN}$ such that $\sum_{n,i} |X_{n,i}-Y_{n,i}| < \infty$ almost surely.
  \end{enumerate}
\end{dfn}

While Berkes \& Rosenthal consider the ambient probability algebra as fixed (this is in particular apparent in their definition of a determining sequence), the model theoretic setting suggests that we allow it to vary.
Conveniently, this has no effect on the definitions:
\begin{fct}
  \label{fct:DeterminingSequence}
  A sequence $(\bar X_n)$ is determining in some $\sC \geq \sC_0 = \sigma\bigl( \{\bar X_n\}_n \bigr)$ if and only if it is determining in $\sC_0$.
  Therefore, from now we just say that a sequence is \emph{determining}.
\end{fct}
\begin{proof}
  Follows from the fact that if $\vec \mu_n$ are conditional distributions over $\sC_0$ and $\vec \mu$ a conditional distribution over $\sC \supseteq \sC_0$, then $\vec \mu_n \to \vec
  \mu$ weakly as conditional distributions over $\sC$ if and only if $\vec \mu$ is in fact over $\sC_0$ and $\vec \mu_n \to \vec \mu$ weakly as conditional distributions over $\sC_0$.
\end{proof}

By \autoref{thm:TypesAreCondDist}, $(\bar X_n)_n$ is determining (in $\sC$, say) if and only if the sequence $\bigl( \tp(\rho \bar X_n/\sC) \bigr)_n$ converges in $\tS_m(\sC)$ to some $\tp(\rho \bar Y/\sC)$, where $\bar Y$ is $\bR^m$-valued random variables (we recall that $\rho\colon \bR \rightarrow (0,1)$ is a homeomorphism fixed throughout).
On the other hand, if we only know that $\bigl( \tp(\rho \bar X_n/\sC) \bigr)_n$ converges in $\tS_m(\sC)$, say with limit $\tp(\bar Z/\sC)$, then $\bar Z$ consists of $[0,1]^m$-valued random variables, so $\bar Y = \rho^{-1} \bar Z$ consists of $[-\infty,\infty]^m$-valued random variables, which need not necessarily be $\bR^m$-valued.

\begin{lem}
  \label{lem:DetSeq}
  Let $(\bar X_n)_n$ be an $\bR^m$-valued sequence, and let $\sC \supseteq \sigma(\{\bar X_n\}_n)$.
  Then the sequence is determining if and only if:
  \begin{enumerate}
  \item The sequence $\bigl( \tp(\rho \bar X_n/\sC) \bigr)_n$ converges in $\tS_m(\sC)$; and:
  \item The sequence $(\bar X_n)_n$ is bounded in measure.
  \end{enumerate}
\end{lem}
\begin{proof}
  Immediate from \autoref{lem:TightConvergentSequence}.
\end{proof}

\begin{prp}[{\cite[Theorem~2.2]{Berkes-Rosenthal:AlmostExchangeableSequences}}]
  Every sequence of $\bR^m$-valued random variables which is bounded in measure has a determining sub-sequence.
\end{prp}
\begin{proof}
  Immediate from \autoref{cor:TightConvergentSubSequence}.
\end{proof}

Clearly, a sequence $(\bar X_n)_n$ is exchangeable if and only if it is an indiscernible set (or more precisely, if and only if the $(0,1)^m$-valued sequence $(\rho \bar X_n)_n$ is).
Since $ARV$ is a stable theory, every indiscernible sequence is indiscernible as a set, so exchangeable is synonymous with indiscernible.
This observation is part of the statement of \cite[Theorem~1.1]{Berkes-Rosenthal:AlmostExchangeableSequences}.
The full statement is that an indiscernible sequence of random variables is conditionally i.i.d.\ over its tail field.
This can also be obtained as an application to $ARV$ of the following facts:
\begin{enumerate}
\item In a stable theory, if $(\bar a_n)_n$ is any indiscernible sequence, then it is a Morley sequence over its ``tail closure'' $C = \bigcap_n \dcl^{eq}\bigl( \{\bar a_k\}_{k \geq n} \bigr)$ (follows from  \cite[Theorem~5.5]{BenYaacov:StableGroups}).
\item The characterisation of $\dcl(A)$ in $ARV$ as $L^1\bigl( \sigma(A), [0,1] \bigr)$.
\item In models of $ARV$, canonical bases exist in the real sort.
\end{enumerate}

On the other hand, being almost exchangeable is \emph{not} invariant under a homeomorphism of $\bR$ with $(0,1)$, so something needs to be said.
Recall first that inside a bounded family of random variables, convergence in $L^p$ is equivalent, for any $1 \leq p < \infty$, to convergence in measure.

\begin{lem}
  \label{lem:IndiscExchgSubSeq}
  Let $\rho\colon \bR \to (0,1)$ be any homeomorphism, and let $(X_n)$ and $(Y_n)$ be sequences of $\bR$-valued random variables.
  \begin{enumerate}
  \item
    If $\sum |X_n - Y_n| < \infty$ a.s.\ then $| \rho X_n - \rho Y_n| \to 0$ in $L^1$.
  \item
    Assume conversely that $| \rho X_n - \rho Y_n| \to 0$ in $L^1$, and that the sequence $(Y_n)_n$ is bounded in measure.
    Then there exists a sub-sequence for which $\sum |X_{n_k} - Y_{n_k}| < \infty$ a.s.
  \end{enumerate}
\end{lem}
\begin{proof}
  For the first item, notice that $\rho$ is necessarily uniformly continuous.
  If $\sum |X_n - Y_n| < \infty$ a.s.\ then by standard arguments $|X_n - Y_n| \to 0$ in measure, in which case $|\rho X_n - \rho Y_n| \to 0$ in measure.
  Since $|\rho X_n - \rho Y_n|$ are bounded random variables, $\| \rho X_n - \rho Y_n \|_1 \to 0$.

  For the second item we assume that $| \rho X_n - \rho Y_n | \to 0$ in $L^1$, or equivalently, in measure, and that $(Y_n)_n$ is bounded in measure.
  We first claim that $|X_n - Y_n| \to 0$ in measure.
  Indeed, let $\varepsilon > 0$ and let $R \in \bR$ be such that $\bP[ |Y_n| > R ] < \varepsilon$ for all $n$.
  Let $K_0 = \rho\bigl[ [-R,R] \bigr]$, $K_1 = \rho\bigl[ [-R-\varepsilon,R+\varepsilon] \bigr]$.
  Then $K_0 \subseteq K_1^\circ \subseteq K_1 \subseteq (0,1)$, and since $K_1$ is compact, $\rho^{-1}$ is uniformly continuous on $K_1$.
  In particular, there exists $\delta > 0$ such that if $x,y \in K_1$ and $|x-y| < \delta$ then $|\rho^{-1}(x)-\rho^{-1}(y)| < \varepsilon$.
  Possibly taking a smaller $\delta$, we may assume that $K_1$ contains a $\delta$-neighbourhood of $K_0$.
  Since $| \rho X_n - \rho Y_n | \to 0$ in measure, for $n$ big enough we have $|\rho X_n - \rho Y_n| < \delta$ outside a set of probability $\varepsilon$.
  Thus, outside a set of probability $2\varepsilon$ we have both $|\rho X_n - \rho Y_n| < \delta$ and $\img (\rho Y_n) \subseteq K_0$, whereby $\img(\rho X_n) \subseteq K_1$ and therefore $|X_n - Y_n| < \varepsilon$.
  This concludes the proof that $|X_n - Y_n| \to 0$ in measure.
  It follows that for a sub-sequence, $\sum |X_{n_k} - Y_{n_k}| < \infty$ a.s.
\end{proof}

\begin{lem}
  \label{lem:IndiscExchgSubSeq2}
  Let $(\bar X_n)_n$ be a sequence of $\bR^m$-valued random variables.
  \begin{enumerate}
  \item
    Assume that $(\bar X_n)$ is almost exchangeable.
    Then it is bounded in measure.
  \item
    Assume that $(\bar X_n)$ is bounded in measure.
    Then it has an almost exchangeable sub-sequence if and only if the sequence $(\rho \bar X_n)_n$ has an almost indiscernible one.
    Moreover, in that case, the indiscernible sequence witnessing almost indiscernibility is $(0,1)^m$-valued.
  \end{enumerate}
\end{lem}
\begin{proof}
  For the first item, it is clear that an exchangeable (and more generally, an identically distributed) sequence is bounded in measure.
  Assume now that $(\bar Y_n)_n$ witnesses that $(\bar X_n)_n$ is almost exchangeable.
  As in the proof of \autoref{lem:IndiscExchgSubSeq} we have $\bar X_n - \bar Y_n \to 0$ in measure, and the statement follows.

  We now prove the second item.
  For left to right, we may assume that $(\bar X_n)_n$ is almost exchangeable, as witnessed by $(\bar Y_n)_n$.
  By \autoref{lem:IndiscExchgSubSeq} the sequence $(\rho \bar Y_n)_n$ witnesses that $(\rho \bar X_n)_n$ is almost indiscernible.

  For right to left, we may assume that $(\rho \bar X_n)_n$ is almost indiscernible, as witnessed by an indiscernible sequence $(\rho \bar Y_n)_n$ (where the $Y_{n,i}$ are, \textit{a priori}, $[-\infty,\infty]$-valued).
  Let $\sC = \sigma\bigl( \{\bar X_n\}_n \bigr)$.
  Then the limit $p = \lim \tp(\rho \bar Y_n/\sC)$ exists, whereby the limit $\lim \tp(\rho \bar X_n/\sC) = p$ exists as well.
  Let $\rho \bar Y \vDash p$.
  By \autoref{lem:DetSeq} $(\bar X_n)_n$ is determining and $\bar Y$ is $\bR^m$-valued.
  Since $\bar Y_n \equiv \bar Y$ (over $\emptyset$, even though not necessarily over $\sC$), each $\bar Y_n$ is $\bR$-valued as well.
  Now, again by \autoref{lem:IndiscExchgSubSeq}, there exist sub-sequences $(\bar X_{n_k})_k$,  $(\bar Y_{n_k})_k$ such that $\sum |X_{n_k,i}-Y_{n_k,i}| < \infty$ a.s.
\end{proof}

Finally, a word regarding the limit tail algebra of a determining sequence.
Let $M = L^1(\sF,[0,1])$ be a big saturated model of $ARV$, $\sC \subseteq \sF$ a sub-algebra, and let $(\bar X_n)_n$ be a determining sequence of $\sC$-measurable random variables.
Let $\bar Y$ realise the limit distribution over $\sC$, measurable in $\sF$ (although not necessarily in $\sC$).
Then the tail measure algebra of $(\bar X_n)_n$ is precisely $\sA
= \sigma\bigl( \{ \bE[(\rho\bar Y)^\alpha|\sC] \}_{\alpha\in\bN^m} \bigr)
\subseteq \sC$, which is interdefinable with $\Cb(\rho\bar Y/\sC)$.

Now, the Main Theorem of \cite{Berkes-Rosenthal:AlmostExchangeableSequences} follows as a special case of our \autoref{thm:GenBR}.

\begin{thm}[{\cite[Main Theorem (2.4)]{Berkes-Rosenthal:AlmostExchangeableSequences}}]
  Let $(\bar X_n)_n$ be a sequence of random variables in a probability space $(\Omega,\sC,\mu)$.
  Then $(\bar X_n)_n$ has an almost exchangeable sub-sequence if and only if it has a determining sub-sequence whose conditional distributions (with respect to the limit tail algebra of the sequence), relative to every set of positive measure, converge strongly.

  Moreover, if in addition $(\bar X_n)$ is determining then the sequence witnessing almost exchangeability is i.i.d.\ over the limit tail algebra.
\end{thm}
\begin{proof}
  We may view $\sC$ as a sub-algebra of a rich atomless probability algebra $\sF$ and work in $M = L^1(\sF,[0,1]) \vDash ARV$.
  Since almost exchangeable and determining sequences are bounded in measure, we may assume that $(\bar X_n)_n$ is bounded in measure.

  Under this assumption, the first condition is equivalent to saying that $(\rho\bar X_n)_n$ admits an almost indiscernible sub-sequence.
  Regarding the second condition, a sub-sequence $(\bar X_{n_k})_k$ is determining if and only if $\bigl( \tp(\bar X_{n_k}/\sC) \bigr)_k$ converge to some type $p \in \tS_m(\sC)$.
  In this case the limit tail algebra $\sA$ is interdefinable with $\Cb(p)$, and the conditional distributions $\dist(\bar X_{n_k},S|\sA)$ converge strongly for every $S \in \sC$ if and only if $\tp( \rho\bar X_{n_k}, \sC/\sA)$ converge in $(\tS(\sA),\sT_\Cb)$, or equivalently, in $(\tS(\sA),\sT_d)$.

  Thus the statement of the theorem is equivalent to saying that the sequence $(\rho\bar X_n)_n$ has an almost indiscernible sequence if and only if it has a sub-sequence with the property $(*_\sC)$.
  This is just a special case of \autoref{thm:GenBR} (and the same for the moreover part).
\end{proof}

\begin{cor}[{\cite[Theorem~3.1]{Berkes-Rosenthal:AlmostExchangeableSequences}}]
  A sequence of random variables has an almost i.i.d.\ sub-sequence if and only if it has a sub-sequence whose distributions relative to any set of positive measure converge to the same limit.
\end{cor}
\begin{proof}
  Let $(\bar X_n)_n$ be the sequence, and let $\sC$ denote the ambient probability algebra in the statement, and we may embed $\sC$ in a model $M \vDash ARV$.

  Following the same translation as above, if $(\bar X_n)_n$ is almost i.i.d., say as witnessed by $(\bar Y_n)_n$, then $\lim \tp\bigl( \rho\bar X_n/\sC) = \lim \tp\bigl( \rho\bar Y_n/\sC) = p$, say, and $\Cb(p) \subseteq \dcl(\emptyset)$.
  In other words, if $\rho\bar Z \vDash p$ then $\sC \ind \bar Z$, meaning precisely that the distribution of $\bar Z$ relative to any non zero member of $\sC$ is the same.

  Conversely, assume that $\lim \dist (\bar X_n|S) = \lim \dist(\bar X_n) = \mu$, say, for every $0 \neq S \in \sC$.
  Let $\bar Z$ realise $\mu$ independently of $\sC$.
  Then $\dist(\bar X_n|\sC) \to \dist(\bar Z|\sC)$ weakly, so the sequence is determining, and since $\bar Z \ind \sC$ the limit tail algebra is trivial.
  Also, since $\dist(\bar X_n|\sC) \to \dist(\bar Z|\sC)$ weakly, we have $\dist(\bar X_n,\sC) \to \dist(\bar Z,\sC)$ (weakly or strongly, over the trivial algebra it is the same thing), so passing to a sub-sequence we may assume that $(\bar X_n)_n$ is almost exchangeable, say witnessed by $(\bar Y_n)_n$.
  By the moreover part of the theorem, this sequence is i.i.d.
\end{proof}

\bibliographystyle{begnac}
\bibliography{begnac}

\end{document}